\documentclass [12pt,a4paper,reqno]{amsart}
\usepackage[foot]{amsaddr}
\usepackage[foot]{amsaddr}
\newcommand\Tref[1]{{Theorem~\ref{#1}}}
\newcommand\Lref[1]{{Lemma~\ref{#1}}}
\newcommand\Cref[1]{{Corollary~\ref{#1}}}
\newtheorem{thm}{Theorem}[section] 

\newtheorem{cor}[thm]{Corollary}

\newtheorem{defn}[thm]{Definition}

\newtheorem{exmpl}[thm]{Example}

\newtheorem{lem}[thm]{Lemma}
\newtheorem{prop}[thm]{Proposition}

\newtheorem{rem}[thm]{Remark}
\newtheorem{notation}[thm]{Notation}

\newcommand{\etype}[1]{\renewcommand{\labelenumi}{(#1{enumi})}}
\def\eroman{\etype{\roman}}
\newcommand\Dref[1]{{Definition~\ref{#1}}}
\newcommand\Rref[1]{{Remark~\ref{#1}}}

\newcommand\M[1][d]{{\operatorname{M}_{#1}}} 
\def\CI{\mathcal I}
\def\CJ{\mathcal J}

\def\({\left(}
\def\){\right)}

\def\bM{\overline{\mathcal M}}
\def\MG{\mathcal M}
\def\sgn{\operatorname{sgn}}

\newcommand{\smat}[4]{{\(\!\!\begin{array}{cc}{#1}\!&\!{#2}\\[-0.1cm]{#3}\!&\!{#4}\end{array}\!\!\)}} 

\def\id{\operatorname{id}}

\def\tr{\operatorname{tr}}

\def\a{{\alpha}}
\def\q{Q}

\def\alphajk{\a_{\operatorname{mat},k}}

\def\aqpol{{\operatorname{\alpha}_{\operatorname{pol}}}^{\bar q}}

\def\la{{\lambda}}

\def\Ann{\operatorname{Ann}}

\newcommand{\set}[1]{{\left\{#1\right\}}}

\newcommand\eq[1]{{(\ref{#1})}}
\newcommand{\card}[1]{{\left|{#1}\right|}}

\def\cha{\operatorname{char}}

\def\II{{I\!\!\,I}}

\def\N {{\mathbb {N}}}

\DeclareMathOperator{\mychar}{char} %
\long\def\forget#1\forgotten{}
\newcommand\suchthat{{\,:\ \,}}

\newcommand\comp[3][\bullet]{{{#1}_{{\if1#2{}\else{#2}\fi}{\if#3K{}\else{(#3)}\fi}}}} 

\newif\ifXY 
\XYtrue     
\ifXY \usepackage{xy}\fi %
\ifXY \xyoption{matrix}\xyoption{arrow}\xyoption{curve} \fi

\def\Zcd{{Zariski closed}}
\def\Zcr{{Zariski closure}}

\usepackage[usenames]{color}

\begin{document}

\title[Representability of relatively free affine algebras]
{Representability of relatively free affine algebras over a Noetherian ring}

\author{Alexei Kanel-Belov$^*$}
\author{Louis Rowen$^*$}
\author{Uzi Vishne$^*$}

\address{$^{*}$ Department of Mathematics, Bar-Ilan University, Ramat-Gan
52900,Israel} %
\email{belova@math.biu.ac.il, rowen@math.biu.ac.il,
vishne@math.biu.ac.il}

\thanks{This research was supported by the Israel Science Foundation [grant number 1994/20].}

\begin{abstract} Over the years questions have arisen about T-ideals of (noncommutative) polynomials.
But when evaluating a noncentral polynomial in subalgebras of matrices, one often has
little control in determining the specific evaluations of the polynomial. One way of
overcoming this difficulty in characteristic 0, is to reduce to multilinear polynomials and utilizing the representation
theory of the symmetric group.  But
this technique is unavailable in characteristic $p>0$.

An alternative method, which succeeds, is the process of ``hiking'' a polynomial,
 in which one specializes its indeterminates in several stages, to obtain a polynomial that contains
  Capelli polynomials,   in order to get control on its evaluations. This
method was utilized on homogeneous polynomials in the proof of Specht's conjecture for affine algebras over fields of positive
characteristic.

In this paper we develop hiking further to nonhomogeneous polynomials, to apply to the representability question.
 Kemer proved in 1988  that every affine relatively free PI algebra
over an infinite field, is representable. In 2010, the first author
of this paper proved more generally that every affine relatively
free PI algebra over any commutative Noetherian unital ring is
representable. We present a different, complete,
proof, based on   hiking nonhomogeneous polynomials, over finite fields. We
then obtain the full result over a Noetherian commutative ring, using Noetherian induction on T-ideals.


The bulk of the proof is for the case of a  base field of positive
characteristic. Here, whereas the usage of hiking is more direct than in proving Specht's conjecture,
one must consider nonhomogeneous polynomials when the base ring is
finite, which entails certain difficulties to be overcome.

In the appendix we show how hiking can be adapted to prove the involutory versions, as well as  various graded  and
nonassociative theorems.
\end{abstract}


\maketitle

\setcounter{tocdepth}{3}
\tableofcontents

\newcommand\AR[1]{{\begin{matrix}#1\end{matrix}}}

\section{Introduction}

In this paper, starting with the hiking technique from \cite{BRV5}, we give a full
exposition of Theorem~\ref{4.66},  that relatively free affine
PI-algebras over an arbitrary commutative Noetherian ring are
representable. The crux of the matter is for algebras over a finite field.
To prove that an affine algebra $A$ is representable, it is enough to embed it into an algebra that is a finite
module over a commutative Noetherian base ring. (Anan'in \cite{An} proved a far more general result, extended even further in
\cite{RowenSmall}.)

Our approach in this representability theorem follows the Shirshov program of \cite[Section~2.4]{BKR}
If the algebra $A$ is integral over its center, then Shirshov's celebrated theorem \cite[Theorem~2.2.2]{BKR} shows that $A$  is finite and then representable.
There is a celebrated method, due to Razmyslov and Schelter, of adjoining characteristic values of the elements of a prime PI-algebra $A$ in order to make it integral,
and the challenge in the proof is to adjoin characteristic values to arbitrary PI-algebras, by means of a polynomial.
(In the prime PI-case, this is the Capelli polynomial of \cite[Definition~1.20]{BKR}.) This might seem circular,
since one obtains characteristic values in matrices, and we want to prove that the algebra is representable; actually we take a maximal representable
T-ideal $\mathcal I$ of $A$, and compute in $A/\mathcal I$. Then we ``hike'' a given identity of $A$ which is not an identity of $A/\mathcal I$ until it sufficiently resembles
a Capelli polynomial that we can adjoin the   characteristic values and obtain our finite module.
The hiking process is summarized in the Canonization Theorem for Polynomials (Theorem~\ref{hikthm}):

Suppose $f(x_1, \dots, x_\ell) $ is a 
 nonidentity of
$\widetilde {A_0}$. Then the T-ideal of~$f$ contains   a critical
non-identity of $\widetilde {A_0}$  (defined in \Dref{doccr}, which involves Capelli-like polynomials).

This idea first appeared in \cite{B2}, which provides the
basis for our proof, but  several aspects remained to be worked out.
Here are the key ingredients in this paper:

\begin{itemize}
\item Description of ``full quivers'' of algebras, which involves collecting ``Canonization theorems'' that
 provide quivers with special properties, in \S\ref{canonth}. Some of the more esoteric sorts
of quivers in \cite{BRV5} can be avoided, since, having already
verified Specht's conjecture (in the affine case) we have more
control over the given T-ideal.

\item The main tool in utilizing the combinatorics of
polynomials is ``hiking,'' to provide special properties of a
polynomial in a given T-ideal.
 The
discussion here of hiking is the main contribution of this paper.
The procedure here is more involved than in \cite{BRV5} since it
involves nonhomogeneous polynomials, cf.~Theorem~\ref{hikthm7}, and is described below in
several stages.

\item  Hiking is coordinated with Noetherian induction on T-ideals,
in the sense that one takes a minimal non-representable
counterexample, and then reduces it further to prove that it is
representable after all. This is to be done by with the help of Shirshov's theorem, by means of adjoining characteristic
coefficients of matrices. But this further reduction is quite
delicate, due to ambiguity in defining the characteristic
coefficients.

On the one hand we want to use a famous trick of
Amitsur (Lemma~\ref{q1}) to extract the characteristic coefficients
from an alternating polynomial $f$, but when our polynomial $f$ is
nonhomogeneous and involves alternating components of different
multiplicities, we cannot apply Lemma~\ref{q1} directly, and
need first to hike $f$, cf.~\S\ref{trab}. It is
convenient to tie them by a different matrix method of calculating
characteristic coefficients, given in Definition~\ref{trmat0}.

\item  In order to identify the matrix action with polynomials, we must
utilize a Noetherian module rather than an algebra, as explained in \S~\ref{stillfield},  and modding out
by this module enables us to apply Spechtian induction.

  \item  The reduction from the Noetherian case to the field case is a
relatively straightforward application of classical Noetherian
induction, included for completeness in \S\ref{tors}.
\end{itemize}

Further applications should be possible in
varied settings, including the proof of an involutory version as well as  various graded theorems and
nonassociative theorems, cf. the appendix (\S\ref{invo}).

 We work with algebras over a commutative Noetherian ring $C$, often a field $F$, with special emphasis on
the possibility that $F$ is finite. $\bar F$ denotes the algebraic
closure of $F$. A finitely generated algebra is called
\textbf{affine}. A~\textbf{(noncommutative) polynomial} is an
element of the free associative algebra~$C\{ x\}$ on countably many
generators. A~\textbf{polynomial identity} (PI) of an algebra~$A$
over $C$ is a noncommutative polynomial  which vanishes identically
for any substitution in~$A$. We use \cite{Row0}, \cite{BKR} as a general reference
for~PIs. A~T-\textbf{ideal}  of $C\{ x\}$  is an ideal $\CI$ of $C\{
x\}$ closed under all algebra endomorphisms $C\{ x\}\to C\{ x\}$. We
write $\id (A)$ for the T-ideal of PIs of an algebra~$A$.

 Conversely, for any T-ideal $\CI$ of $C\{ x\}$, each element of
$\CI$ is a PI of the algebra  $C\{ x\}/\CI$, and $C\{ x\}/\CI$ is
\textbf{relatively free}, in the sense that for any PI-algebra $A$
with $\id(A) \supseteq \CI,$ and any $a_1, a_2,\ldots\in A,$ there
is a natural homomorphism $C\{ x\}/\CI \to A$ sending $x_i \mapsto
a_i$ for $i = 1,2, \dots .$

When $A = C\{ x\}/\CI$ is relatively free, and $\overline{\CJ} = \CJ
/ \CI$ for a T-ideal $\CJ \supset  \CI$  of $F\{ x\}$, we also call
$\overline{\CJ} $ a T-ideal of $A$. ($\overline{\CJ}$ is invariant
under all endomorphisms of $A$.)


\subsection{Representability}$ $

We mostly follow \cite[\S 1.6]{BKR} and \cite{RowenSmall}. An
algebra $A$ over a field $F$ is called {\bf representable} if it is
embeddable as an $F$-subalgebra of $\M[n](K)$ for a suitable field
$K \supseteq F$.

An algebra $A$ over a commutative ring $C$ is called {\bf weakly
representable} if it is embeddable as a $C$-subalgebra of a finite dimensional
algebra over a commutative Noetherian $C$-algebra $K$. This
definition is weaker than \cite[Definition 1.6.1]{BKR} in order to
avoid {Bergman's} example  \cite{Bergman} of a finite ring not
embeddable into matrices over a commutative ring; any finite ring is
weakly representable in our sense.

Obviously any representable algebra is
weakly representable. On the other hand, by~\cite{RowenSmall} any
Noetherian algebra over a field which is finite over its center is
representable, so ``representable'' and ``weakly representable''
coincide for algebras over a field.

Any weakly representable algebra is PI, but an easy counting
argument of Lewin~\cite{Lew} leads to the existence of
non-representable affine PI-algebras over any field.

Nevertheless, the representability question for relatively free
affine algebras has considerable independent interest, and  the
purpose of this paper is to give a full proof of the following
results:

\begin{thm}\label{4.66}
Every relatively free affine PI-algebra over an arbitrary  field is
representable,
\end{thm}

and, more generally,

\begin{thm}\label{4.67}
Every relatively free affine PI-algebra over an arbitrary
commutative Noetherian ring is  weakly representable.
\end{thm}

Kemer obtained Theorem~\ref{4.66} over infinite fields by means of the following amazing results:

\begin{thm}[{\cite[Theorem~6.3.1]{BKR}, \cite{Kem11}}]\label{fd1}$ $
\begin{enumerate}
\item Every affine PI-algebra over an infinite field (of arbitrary characteristic) is PI-equivalent  to a
finite dimensional (f.d.)~algebra.
\item Every   PI-algebra of characteristic 0 is PI-equivalent  to the Grassmann envelope of a
finite dimensional (f.d.)~algebra.
\end{enumerate}
\end{thm}

An immediate consequence of Theorem~\ref{fd1}(1) is that every
relatively free affine PI-algebra over an infinite field is
representable, since it can be constructed with generic elements
obtained by adjoining commutative indeterminates to the
f.d.~algebra.

\begin{rem}\label{Noethind2} For graded associative algebras \cite{AB} and various nonassociative affine algebras of characteristic 0,
the finite basis of T-ideals has been established in the case when
the operator algebra is PI (Iltyakov~\cite{Ilt91,Ilt92} for
alternative and Lie algebras, and Vais and Zelmanov \cite{VZ} for
Jordan algebras) but in many cases the representability question for relatively free
affine algebras
remains open for nonassociative algebras, so   presumably is more difficult than the
finite basis of T-ideals. The obstacle is getting started via some
analog of Lewin's theorem \cite{Lew}, which often is not yet
available. Belov~\cite{B5} obtained representability of   relatively free
 alternative
or Jordan algebras satisfying all identities of some finite
dimensional algebra.
\end{rem}

\subsection{Overview of the proof of Theorem~\ref{4.66}}$ $

  Kemer deduced from Theorem~\ref{fd1} the solution of Specht's problem (the finite basis of T-ideals) for an affine algebra over an infinite
field, in which he applied combinatorial techniques to representable
algebras. The approach here for positive characteristic is the
reverse, since Theorem~\ref{fd1} can fail over finite fields.

In~\cite{BRV1}--\cite{BRV3}, \cite{BRV5}, summarized in \cite{BRV6}, we have provided a complete proof for
the affine case of Specht's problem  in arbitrary characteristic
(The non-affine case has counterexamples, cf.~\cite{B0,B1}):

\begin{thm}[\cite{BRV5}]\label{Spaff}
Any affine PI-algebra over an arbitrary commutative Noetherian ring
satisfies the ACC on T-ideals.
\end{thm}

So we may start with the solution of Specht's problem and apply
Noetherian induction to prove representability of affine relatively
free PI-algebras over  finite fields. Together with
Theorem~\ref{fd1}, we then have Theorem~\ref{4.66}. (These methods
also work in characteristic~0, but then rely on Kemer's solution of
Specht's problem in characteristic~0, which in turn requires his
representability theorem in characteristic~0.)

Then we apply facts about torsion in rings, collected in \S\ref{467}, to obtain
Theorem~\ref{4.67}.

 \section{Preliminaries to the proof of Theorems~\ref{4.66} and ~\ref{4.67} }

We fix the following notation: We start with  the free associative affine algebra $C \{ x \} = C\{ x_1,
\dots, x_\ell\}$ in $\ell$ indeterminates, and a T-ideal $\CI$.
This
gives us the relatively free algebra
$$A = C \{ x \}/\CI.$$  We say that the T-ideal $\CI$ is
\textbf{(weakly) representable} if the affine algebra $A $ is (weakly) representable.

\subsection{The underlying approach}$ $

Note that the direct sum of two  weakly
 representable algebras is  weakly representable, and the direct sum of two
 representable algebras over a field is   representable.

\begin{rem}\label{Noethind0}
The proof of Theorems~\ref{4.66} and ~\ref{4.67} goes along the
following version of Noetherian induction:

We shall show that every T-ideal $\CI$ is weakly representable, and is
representable when $C$ is a field. Taking a T-ideal $\CI$ maximal
with respect to $A$ not weakly representable, we call  $A$ a
\textbf{Specht minimal counterexample}.

For $C$ a field $F$, a key technique is to embed $A$ into a direct
sum of two relatively free algebras, one of which is a homomorphic
image of $A$ and thus representable by Specht induction, and the
other of which is representable by some structural argument.
\end{rem}

We call this argument  \textbf{Specht induction}. It was utilized by
Kemer in his proof of Specht's conjecture, cf.~\cite[Proposition
6.6.31]{BKR} and can be treated abstractly, applicable also to
nonassociative algebras:

\begin{lem}\label{Phoenix}
Suppose that the
relatively free $F$-algebra $A$ is a Specht minimal counterexample
to representability, and~$f$ is a polynomial generating a T-ideal $\langle f \rangle_T$ which  annihilates some
T-ideal of~$A$, and $A/\mathcal I$ is representable for some T-ideal
$\mathcal I$ for which $\mathcal I \cap \langle f \rangle_T = 0.$ Then we have a
contradiction to $A$ being a Specht minimal counterexample.
\end{lem}
\begin{proof}
We can embed $A$ into $(A/\mathcal I)\oplus \langle f \rangle_T.$ But $\langle f \rangle_T$ is
representable by Specht induction, so $A$ is representable, contrary
to assumption.
\end{proof}

The bulk of this paper consists of the proof of Theorem~\ref{4.66}.
From now on,  $A$ is assumed to be a Specht minimal counterexample.
The proof  relies on the ideas of the proof of Kemer's
representability theorem of relatively free
affine algebras given in \cite{BR,BKR}. Much of this paper
is devoted to elaborating the theory of \cite{BRV2} and \cite{BRV3}
in the field-theoretic case, as described in \cite{BRV5,BRV6}, and
there is a considerable overlap with~\cite{B2}.

Until \S\ref{467}, we work over a field $F$. Since the result is
known in characteristic~0, we assume through
Section~\ref{stillfield} that $F$ has characteristic $p>0.$

 \begin{rem}\label{Noethind}
In view of Lewin's theorem~\cite{Lew}, any
T-ideal $\CI$ contains a representable T-ideal, which by
Theorem~\ref{Spaff} is contained in a maximal representable T-ideal
$ \CI_0$ of $A$ contained in $\CI$, which we aim to show is equal
to~$\CI$. Thus $A_0: = F \{ x \}/\CI_0$ is representable (and affine
over $F$). Assuming   that $\CI_0 \ne \CI,$ we have reduced to the
case where $A_0$ is representable but every nonzero T-ideal of~$A_0$
contained in $\CI$ and properly containing $ \CI_0$ is not
representable.
%
%

Our goal is to arrive at a contradiction by starting with a
polynomial  $f \in \CI \setminus \CI_0$ (in other words, a
non-identity of $A_0$ in $\CI$) and adjusting $f$ to a polynomial
$\tilde f \in \CI \setminus \CI_0$, such that the T-ideal $\CI_1
\subseteq \CI$
 generated by $\CI_0$ and $\tilde f$ is a representable T-ideal.  In this manner, we do
not need to introduce parameters of induction, as opposed to Kemer's
approach described in \S\ref{Kmed}.
\end{rem}

In our proof, we start with $A_0 = F \{ x \}/\CI_0$ of
Remark~\ref{Noethind}.

\begin{notation}\label{not1}  Being representable,
 $A_0 \subset
\M[n](K)$ with $K$ algebraically closed, and we fix this particular
representation. The ``Zariski closure'' $\widetilde {A_0}$ of $A_0$
in $\M[n](K)$ with respect to the Zariski topology
\cite[\S~3.1]{BRV6} is PI-equivalent to $A_0$, and  so we work
throughout with ~$\widetilde {A_0}$. We emphasize that $\CI_0$ is
the  T-ideal of identities of $\widetilde {A_0}$ as well as of~$A_0$.
(When $F$ is infinite then we may assume that $K=F$,
cf.~\cite[Remark~3.1]{BRV6}, but the situation for $F$ finite is
more delicate.) By the version of Wedderburn's Principal Theorem
\cite[Theorem 2.5.37]{Row1.5}, $\widetilde {A_0} = S \oplus J$ as
vector spaces, where $J$ is the radical of $\widetilde {A_0}$ and $S
\cong \widetilde {A_0}/J$ is a semisimple subalgebra of $\widetilde
{A_0}$. Thus $S$ is a direct product of matrix algebras $R_1 \times
\dots \times R_k$, called \textbf{Wedderburn blocks}, which we want
to view along the diagonal of $\M[n](K)$, although  possibly with
some identification of coordinates, which are to be described
graphically. By the Braun-Kemer-Razmyslov theorem, cf.~\cite{Br},
$J$ is nilpotent, so we take $t:  = t_{\widetilde {A_0}}$ maximal
such that $J^t \ne 0.$ This description of $A$ as  $(R_1 \times
\dots \times R_k)\oplus J$ is called  \textbf{Wedderburn block form}.
\end{notation}

\subsubsection{Multilinearization versus quasi-linearization}$ $

The well-known linearization process of a polynomial can be
described in two stages: First, writing a polynomial $f(x_1, \dots,
x_n)$ as
$$f(0,x_2, \dots, x_n)+ (f(x_1, \dots, x_n)- f(0,x_2, \dots, x_n)),$$ one sees by
iteration that any T-ideal is additively spanned by T-ideals of
 polynomials for which each indeterminate appearing nontrivially
appears in each of its monomials,
 cf.~\cite[Exercise~2.3.7]{Row1}. Then we could define the \textbf{linearization process}
 by
 introducing a new indeterminate $x_i'$ and
 passing to $$f(x_1, \dots, x_i +x_i', \dots, x_m) - f(x_1, \dots, x_i , \dots,
 x_m)-f(x_1, \dots, x_i', \dots, x_m).$$
This process, applied repeatedly, yields a multilinear polynomial in
the same T-ideal. In characteristic 0 the multilinearization process
can be reversed by taking $x_i' = x_i,$ implying that every T-ideal
is generated by multilinear polynomials. But this fails in positive
characteristic, and more generally when integers are not invertible,
as exemplified by the Boolean identity $x^2 -x$, so we need an
alternative. To handle characteristic $p>0$, Kemer ~\cite{Kem2}
considered the following modification, which we review
from~\cite{BRV6}.

\begin{defn}\label{QL}
A polynomial $f\in \CI$ is $i$-\textbf{quasi-linear} on an algebra
$A$ if
$$f(\dots, a_i + a_i', \dots) = f(\dots, a_i , \dots)+f(\dots,
a_i', \dots)$$ for all $a_i, a_i' \in A; $ $f$ is
$A$-\textbf{quasi-linear} if $f$ is $i$-quasi-linear on $A$ for all
$i$.\end{defn}

Suppose $f(x_1, x_2, \dots) \in F\{ x\}$ has degree $d_i$ in $x_i$.
The \textbf{$i$-partial linearization step} of~$f $ is
\begin{equation}\label{partlin}\Delta_i f := f(x_1, x_2, \dots, x_{i,1}+\cdots+ x_{i,d_i}, \dots)- \sum
_{j=1}^{d_i} f(x_1, x_2, \dots, x_{i,j}, \dots)\end{equation} where
the substitutions were made in the $i$ component, and
$x_{i,1},\dots,x_{i,d_i}$ are new variables.

 When $\Delta_if (A)= 0$,
then $f$ is $i$-quasi-linear on $A$, so given a non-identity $f$
of $A$ we apply \eqref{partlin} at most $\deg_i f$ times repeatedly,
if necessary, to each $x_i$ in turn, to obtain a non-identity of $A$
in the T-ideal of $f$, that is quasi-linear.

\begin{prop}[{Special case of \cite[Theorem 1.4]{D}, also cf.~\cite[Corollary~2.13]{BRV4}}]\label{quasi}
Assume $\mychar F = p > 0$. For any non-identity~$f$  of $A$, the
T-ideal generated by $f$ contains a quasi-linear non-identity of
$A$, for which the degree in each indeterminate is a $p$-power.
\end{prop}

We apply all this to $\widetilde {A_0}$. We just say that an
$\widetilde {A_0}$-quasi-linear polynomial $f$ is
\textbf{quasi-linear}.
 In view of Proposition~\ref{quasi} we assume
from now on that our polynomial $f(x_1, \dots, x_n)$ is
quasi-linear. When specializing~$x_i$ to an element $\bar x_i$
of~$\widetilde {A_0}$, we call the substitution $\bar x_i$
\textbf{radical} if $\bar x_i \in J,$ and \textbf{semisimple} if $\bar
x_i \in S.$ The substitution $\bar x_i$ is \textbf{pure} if it is
radical or semisimple. The substitution  $f(\bar x_1, \dots, \bar
x_n)$ of $f(x_1, \dots, x_n)$ is \textbf{pure} if each  $\bar x_i$
is {pure}. Writing any substitution $\bar x_i$ as a sum of radical
and semisimple substitutions,    since $f(x_1, \dots, x_m)$ is
quasi-linear, we can reduce all substitutions in $J+S$ to   pure
substitutions (in $ J\cup R_1 \cup \dots \cup R_k$). In particular,
$f$ has a nonzero specialization where all substitutions $\bar x_i$
are pure.

Any semisimple substitution $\bar x_i$ is  in $S$ and thus in a
block (or in glued blocks) of some degree $n_i$, which we also call
the \textbf{degree} of $\bar x_i$. A~radical substitution $\bar x_i$
is somewhat more subtle. It is viewed as an edge connecting two
vertices in blocks, say of degrees $n_{i_1}$ and $n_{i_2} $. If
these blocks are not glued, then we call this substitution a
\textbf{bridge} of \textbf{degrees} $n_{i_1}$ and $n_{i_2} $. A
bridge is \textbf{proper} if $n_{i_1} \ne n_{i_2} $.

\subsubsection{Folds in a polynomial}$ $

As in Kemer's proof of Specht's conjecture and the proof  given in
\cite{BKR}, our first task is to estimate $d:=[\widetilde {A_0}:F]$
and $t$ in terms of polynomials.
\begin{defn}\label{alt4} A polynomial $f(x_1, \dots, x_n)$ \textbf{$m$-alternates in}
$x_{i_1},\dots, x_{i_m}$ if $$f(\dots,x_{i_1},\dots, x_{i_2},\dots,
x_{i_m},\dots) = \sgn (\pi) f(\dots,x_{\pi(i_1)},\dots,
x_{\pi(i_2)},\dots, x_{\pi(i_m)},\dots)$$ for any $\pi \in S_m$.
\end{defn}

For example, $c_m$  denotes the \textbf{Capelli polynomial} in $2m$
distinct
indeterminates, 
which is alternating in the first $m$ indeterminates (i.e., switches
sign when interchanging two of these indeterminates). Thus,  for any
field $F$, $c_{2m^2}$ is an identity of $\M[m-1](F)$ but not an
identity of $\M[m](F)$.

We need several sets of alternating indeterminates.

\begin{defn}\label{alt3}
A polynomial $f(x_1, \dots, x_n)$
is \textbf{$\mu$-fold $m$-alternating} if $f$ alternates in $\mu$
disjoint sets of indeterminates $\{x_{j_{r1}}, \dots, x_{j_{rm}}\}$,
$1 \le r \le \mu$.
\end{defn}

\begin{rem}\label{fold}
Given a polynomial $f(x_1, \dots, x_n)$, one way of increasing the
number of $m$-alternating folds  is by replacing an indeterminate
$x_i$ occurring linearly in $f$  by $x_i c_m (x_{n+1}, \dots
x_{n+m^2}).$

For example $h_{m,i}(y)$ denotes a multilinear central polynomial
$h_{m,i}(y_{i_1}, \dots, y_{i_m'})$ for $\M[m](F)$, alternating in
specific  indeterminates $y_{i_1}, \dots, y_{i_{m^2}}$  which are
all distinct. This is done by   inserting a fold into a multilinear
central polynomial, cf.~\cite[pp.~37,38]{BKR}.
\end{rem}

\subsection{Comparison with Kemer's method}\label{Kmed}$ $

 The proof of Theorem~\ref{Spaff} in \cite{BRV5} is somewhat
different from Kemer's proof.
 Kemer brought in some combinatoric
definitions:

\begin{defn}   $\beta( {A_0})$ is the largest $\beta$ such that, for any
$\mu$, there is a $\mu$-fold $\beta$-alternating non-identity of $
{A_0}$.

$\gamma( {A_0})$ is the largest $\gamma$ such that, for arbitrarily
large $\mu$ there is a $\mu$-fold, $\beta(\widetilde
{A_0})$-alternating and $(\gamma-1)$-fold, $(\beta(\widetilde
{A_0})+1)$-alternating non-identity of $ {A_0}$. Such a polynomial
is called a $\mu$-\textbf{Kemer polynomial} for $ {A_0}$,
\cite[Definition~6.6.7]{BKR}.

The pair $(\beta( {A_0}),\gamma( {A_0}))$ is called the
\textbf{Kemer index} of $A_0,$ which we order lexicographically,

 The \Zcd\ algebra $\widetilde {A_0}$ is \textbf{full}
 with respect to a monomial $g$ if
    some nonzero
substitution of $g$  passes through all the blocks of the quiver.

A multilinear polynomial $f$ has \textbf{Property K} on a f.d.
algebra $W$ if $f$ vanishes under any specialization with fewer than
$t-1$ radical substitutions.
\end{defn}

In this case, in characteristic 0, Kemer's First
Lemma~(\cite[Proposition~6.5.2]{BKR}) says that, for $F= \bar F$
algebraically closed, if $\widetilde {A_0}$ is full then
$\beta(\widetilde {A_0})  = [\widetilde {A_0}:F]$; Kemer's Second
Lemma~(\cite[Proposition~6.6.31]{BKR}) says that when $\widetilde
{A_0}$ is not PI-equivalent to a finite direct product of algebras
of lower Kemer index, $\gamma(\widetilde {A_0})$ is the index of
nilpotence of $J$, and~$A$ has multilinear $\mu$-Kemer polynomials
for arbitrarily large $\mu$.

Kemer's First and Second Lemma are the keys to Kemer's proof, with
the pair $(\beta(\widetilde {A_0}),\gamma(\widetilde {A_0}))$
forming the basis for induction. But one relies on characteristic 0,
in order to stay within the T-ideal $\CI$ when multilinearizing.  In
characteristic $p$ one only has  quasi-linearization, so we need
some alternative form of induction. Moreover, there is no obvious
way to pass to basic algebras since we are working with $\CI_0$, not
$\CI.$ So we turn to a different method not relying on these
parameters.

\subsubsection{The alternative method of full quivers}$ $

Let us review some of the main techniques we need for the proof. The
reader can refer to \cite{BRV6} for further details. We
 rely on  two
languages: quivers $\Gamma$ of  the representations of $A$ on one
hand, versus the combinatorial language of polynomials on the other
hand.

First we bring in the  language of quivers.
 In~\cite{BRV2} we considered the {\bf
full quiver} of a representation of an associative algebra over a
field, and determined properties of full quivers by means of a close
examination of the structure of the closure under the  Zariski
topology, studied in \cite{BRV1}. Then we modified $f$ by means of a
``hiking procedure'' in order to force $f$ to have certain
combinatorial properties, and used this to carve out a T-ideal
$\mathcal J$ from inside a given T-ideal; modding out $\mathcal J$
lowers the quiver in some sense, and then one obtains Specht's
conjecture by induction. Hiking turns out to be a powerful but
intricate tool.

Our approach here is similar, but with some variation. Here we need
not mod out by $\mathcal J$, but do need $\mathcal J$ to be
representable. We start the same way, but one of the key steps  in
\cite{BRV5} fails since we must cope with non-multilinear
polynomials, and we need a way of getting around it.

%
%

%
\subsection{Review of full quivers}$ $

 One of the most
useful tools in  representation theory is the quiver of a
f.d.~algebra, for which we  present a modification (for the Zariski
closed algebra $\widetilde {A_0}$) more pertinent to PI-theory.

We need an explicit description, but which may distinguish among
Morita equivalent algebras since matrix algebras of different size
are not PI-equivalent. The {\bf full quiver} of ${A_0}$, or of its
Zariski closure $\widetilde {A_0}$, is a directed graph $\Gamma$,
having neither double edges nor cycles, with the following
information attached to the vertices and edges:

The vertices of the full quiver of $\widetilde {A_0}$ correspond to
the diagonal matrix blocks arising in the semisimple part $S$,
whereas the arrows
come from the radical~$J$. Every vertex likewise corresponds to a
central idempotent in a corresponding matrix block of $\M[n](K)$. 

\begin{itemize}
\item The vertices are ordered, say from
$\bf 1$ to $\bf k$, and an edge always takes a vertex to a vertex of
higher order. There are identifications of vertices, i.e., of matrix
blocks, called \textbf{diagonal gluing}, and identification of
edges, called \textbf{off-diagonal gluing}. Gluing of vertices in
full quivers is \textbf{identical}, as in $\set{\smat{\alpha}{*}{0}{\alpha} \suchthat \alpha \in \bar F}$, or \textbf{Frobenius}, as in
$\set{\smat{\alpha}{0}{0}{\alpha^q} \suchthat \alpha \in \bar F}$
 where $\card{F} = q$.
\item Each vertex is labeled
with a roman numeral ($I$, $\II$ etc.); glued vertices are labeled
with the same roman numeral. A
vertex can be either \textbf{filled} or \textbf{empty}. 

The first vertex listed in a glued matrix block is also given a pair
of subscripts --- the \textbf{matrix degree} $n_{\bf i}$ and the
\textbf{cardinality} of the corresponding field extension of $F$
(which, when finite, is denoted as a power $q^{t_{\bf i}}$ of $q
=\card{F}$).
\item Superscripts  indicate the
\textbf{Frobenius twist} between glued vertices, induced by the
Frobenius automorphism $a \mapsto a^q;$ this could identify $a^{q_1}
$ with $a^{q_2}$ for powers $q_1,q_2$ of $q$ (or equivalently $a$
with $a^{q_2/q_1}$ when $q_1 < q_2$); we call this
\textbf{$(q_1,q_2)$-Frobenius gluing}.

\item Off-diagonal gluing (i.e., gluing among the edges) includes Frobenius gluing (which only
exists in nonzero characteristic) and \textbf{proportional gluing}
obtained by multiplying by an accompanying \textbf{scaling
factor}~$\nu$. \textbf{Proportional Frobenius gluing} is Frobenius
gluing combined at the same time with proportional gluing.
\end{itemize}
 Examples are given   in \cite{BRV3}.

%
%
%
%
%

\subsection{Review of the three  canonization theorems for
quivers}\label{canonth}$ $

Since arbitrary gluing is difficult to describe, we need some
``canonization'' theorems to ``improve'' the gluing. The first
theorem shows that we have already specified enough kinds of gluing.

\begin{thm}[{First Canonization Theorem, cf.~\cite[Theorem~6.12]{BRV2}}]\label{main}
The \Zcr\ $\widetilde {A_0}$ of any representable affine PI-algebra
$A_0$ has a representation for whose full quiver every gluing is
proportional Frobenius.
\end{thm}

For the Second Canonization Theorem we grade paths according to the
following rule:

\begin{defn} When $|F| = q < \infty,$ we write $\MG_\infty$ for the multiplicative monoid $\set{1,q,q^2,\dots,\epsilon}$,
where $\epsilon a = \epsilon$ for every $a \in \MG_{\infty}$. (In
other words, $\epsilon$ is the zero element adjoined to the
multiplicative monoid $\langle q \rangle$.)   Let $\bM$ be the
semigroup $\MG_\infty /\!\!\sim$
 where $\sim$ is the equivalence relation obtained by
matching the degrees of glued variables: When two vertices have a
$(q_1,q_2)$-Frobenius twist, we identify $1$ with  $
\frac{q_1}{q_2}$ in the respective matrix blocks, and use $\bM$ to
grade the paths.
\end{defn}

\begin{defn}
A~full quiver is \textbf{primitive} (called \textbf{basic} in
\cite{BRV3}) if it has a unique initial vertex $r$ and unique
terminal vertex $s$, and all of its gluing above the diagonal is
proportional Frobenius. A~primitive full quiver $\Gamma$ is
\textbf{canonical} if
any two paths from the vertex $r$ to the vertex $s$ have the same
grade.
\end{defn}


\begin{thm}[{Second Canonization Theorem, cf.~\cite[Theorem~3.7]{BRV3}}]\label{Can2}
Any relatively free algebra is a subdirect product of algebras whose
full quivers are primitive.

Any primitive full quiver $\Gamma$ of a representable relatively
free algebra can be modified (via a change of base) to a canonical
full quiver.
\end{thm}

In view of this result, we may reduce to the case that the full
quiver of our polynomial~$f$ is primitive.

 The Third Canonization
Theorem \cite[Theorem~3.12]{BRV3} describes what happens when one
mods out a ``nice'' T-ideal, so is not relevant, since all we need
is to find a representable T-ideal, which we do later by another
method.

\section{The Canonization Theorem for Polynomials}$ $

We continue the proof of \Tref{4.66} following the strategy outlined
in \Rref{Noethind}.

We have two languages: quivers and their representations on one
hand, versus the combinatorial language of identities on the other
hand.
 We are given a polynomial $$f(x_1,
\dots, x_m) = \sum g_j (x_1, \dots, x_m)\in \CI \setminus  \CI _0,$$
for monomials $g_j$.

\subsection{The geometric aspect}$ $

First we consider the geometrical aspect, using quivers. A
\textbf{branch} $\mathcal B$ of $f$ is a path that appears in a
nonzero specialization of some monomial of~$g_j$.

 The
 {\bf length} of the branch $\mathcal B$ is its number of arrows,
excluding loops, which equals its number of vertices (say $k$) minus
$1$. Thus, a typical branch has vertices of various matrix degrees
$n_j$, $j = 1, 2, \dots,k$. We call $(n_1, \dots, n_k)$ the
\textbf{degree vector} \cite[Definition~2.32]{BRV5} of the branch
$\mathcal B$. The \textbf{descending degree vector} is obtained by
ordering the entries of the degree vector to put them in descending
order lexicographically (according to the largest $n_j$ which
appears in the distinct glued matrix blocks, excluding repetitions,
taking the multiplicity into account in the case of Frobenius
gluing). We write the descending degree vector as $(\pi(\mathbf n)_1,
\dots, \pi(\mathbf n)_k)$. Thus, $\pi(\mathbf n)_1=
\max\set{n_1,\dots,n_k}$.

We denote the largest $n_j$ appearing  in a nonzero specialization
of the quiver  as $\tilde n$, and fix this substitution $\bar x_1,
\dots, \bar x_n$ for the time being.  Any other substitution is
denoted~$\bar x_i'$. A~proper bridge connecting vertices of degree
$n_i \ne n_j$ is an $\tilde n$-\textbf{bridge} if $n_i$ or $n_j$ is~$\tilde n$. But there also is the possibility that a radical
substitution connects two glued blocks both of the same degree
$\tilde n$, in which case we call it $\tilde n$-\textbf{internal}.

\begin{defn}\label{def21}
A  branch  $\mathcal B$ of $f$ is  \textbf{dominant} if it has the
maximal number of $\tilde n$-bridges, has maximal length $k$ with
regard to this property, and has the maximal number of vertices of
$\tilde n$-bridges among these in the lexicographic
order, 
and then we continue down
the line to $\tilde n -1$, etc. The {\textbf{depth}} of a dominant
branch $\mathcal B$ is the number of times~$\tilde{n}$ appears in
its degree vector.
%

We work with a dominant branch $\mathcal B$ of the quiver $\Gamma$
in~$f$. A branch is
 \textbf{pseudo-dominant} if it has the same configuration of bridges (although perhaps
 with different multiplicity) as  $\mathcal B$.
  We define the  \textbf{pseudo-dominant
 components} of $f$ to be those sums of monomials whose branches are
 pseudo-dominant with the same degrees.
\end{defn}

(This extra complication of {pseudo-dominant components only arises
when the polynomial~$f$ is nonhomogeneous.) Our goal is somehow to
force every nonzero substitution of~$f$ into a pseudo-dominant
branch by
considering each degree in turn from $\tilde n$ down. %
After the first two stages of hiking, given in \S\ref{firstst1},
\S\ref{sechik}, $f$ will contain some term
$$h_m = h_{m,1} g_1 h_{m,2} g_2\cdots g_t h_{m,t+1},$$  the product of $t+1$ copies of distinct  polynomials $h_{m,i}$ of the same degree
$2m^2$; we
 call the $h_{m,i}$ the \textbf{components} of~$h$.
We focus first on semisimple substitutions having matrix degree
$\tilde n,$ and put $h = h_{\tilde n}.$

\begin{lem}\label{expans10} Any nonzero specialization of $h$ has a component consisting solely of  semisimple substitutions (all of the same
degree).
\end{lem}
\begin{proof} Otherwise every component has a radical substitution,
so we have a product of $t+1$ radical elements, which is 0 by
definition of $t$. \end{proof}

Viewing a substitution of $x_i$ as corresponding to an edge in the
quiver, we have two degrees, one for each vertex of the edge.

\begin{defn} Suppose that $m$  is one of the two degrees  of the substitution
$\overline{x_i}$. A substitution $\overline{x_i}'$ of $x_i$ is
\textbf{$m$-right} if  one of the two degrees of  $\overline{x_i}'$ is  $m$;
$\overline{x_i}'$ is \textbf{$m$-wrong} if  both  degrees of
$\overline{x_i}'$ differ from $m $.

 We write \textbf{right} (resp.~ \textbf{wrong})
for $\tilde n$-right (resp.~$\tilde n$-wrong).
\end{defn}

One delicate point:  An internal radical bridge, say from one matrix
block of degree~$m$ to a different matrix block of degree~$m$, is
technically ``$m$-right'' according to this definition, but must be
dealt with separately.

\begin{rem}\label{dombr} $ $ \begin{enumerate}
\item In view of Lemma~\ref{expans10}, a wrong substitution could lead
to~$h$ (and thus~$f$) having a component with semisimple
substitutions in a matrix block of the wrong degree.

\item  Also, we must contend with the possibility that
right substitutions of dominant branches could cancel, thereby not
yielding nonzero evaluations of~$f$.
\end{enumerate}\end{rem}
%
%
%

\begin{rem} Suppose $f(x_1, \dots, x_\ell) $ is a full nonidentity of $\widetilde {A_0}$, via the
dominant branch~$\mathcal B$ say of degrees $m_1, \dots, m_k$ having
some number $k$ of bridges, and $k'$ internal radical substitutions.
By Theorem~\ref{Can2}, any wrong nonzero substitution may be assumed
to have $k$ bridges since otherwise we may apply induction  to the
number of semisimple components in the full quiver. On the other
hand,  taking the nonzero substitution of $\mathcal B$ with $k'$
maximal, any wrong substitution has at most $k'$ internal radical
substitutions.
\end{rem}

 Our objective is to modify~$f$ to a non-identity of
$\widetilde {A_0}$, containing a Capelli component which enables us
to use combinatorial methods to calculate characteristic
coefficients in a Shirshov extension, with multiplication by
elements of $\widetilde {A_0}$. We prove the following main
  result, enabling us to correspond quivers with properties of
  polynomials, and which leads directly to the representability theorem.

\begin{thm}[Canonization Theorem for Polynomials]\label{hikthm}
Suppose $f(x_1, \dots, x_\ell) $ is a 
 nonidentity of
$\widetilde {A_0}$. Then the T-ideal of~$f$ contains   a critical
non-identity of $\widetilde {A_0}$  (defined in \Dref{doccr}).
\end{thm}

\subsection{Explicit description of the Canonization Theorem for
Polynomials}$ $

The set-up of the Canonization Theorem for Polynomials, based on
``hiking,'' is done in several stages:

\begin{enumerate}
\item Eliminate unwanted semisimple substitutions.
\item Make sure that the remaining substitutions are in the ``correct''
semisimple components.
\item Locate an ``atom'' (see \Dref{mol}) inside the polynomial where we
can compute the action of characteristic coefficients.
 \end{enumerate}

\subsection{The unmixed case}\label{nonmixed}$ $

We quickly dispose first of the following easy case, following
Kemer. We say a substitution is \textbf{unmixed} if it does not
involve any bridges, i.e., all substitutions are in a single
Wedderburn block. Here we need only multiply by a Capelli polynomial
of the matrix degree, and then may proceed directly to the method of
\S\ref{trab}.

Although easy, this aspect is crucial to our proof, since
substitutions alone are not sufficient to take care of examples such
as the non-finitely generated T-space of \cite{Shch00} (generated by
$\{[x_1,x_2]x_1^{p^k-1}x_2^{p^k-1}: \ k \in \N\}$ in the Grassmann
algebra with two generators; also see~\cite{G,G2}).

Furthermore, it provides the base for our induction on $\tilde n$.

%

\subsection{The mixed case: Introducing the hiking procedure}$ $

To attain the proof of the Canonization Theorem for Polynomials, we
must turn to the mixed case.
 In our combinatorics we need to cope with the danger that our
 substitutions are  wrong,  or the base field of the semisimple component
 is of the
wrong size. To prevent this, we insert  substitutions of multilinear
polynomials for indeterminates inside~$f$, called \textbf{hiking},
which force the substitutions to become 0 in such situations. In
other words, hiking replaces~$f$ by a more complicated polynomial in
its T-ideal, which yields a  zero value when we apply a wrong
substitution to the original indeterminates of~$f$. The notion of
hiking passes from branches of quivers to combinatorics of
nonidentities, showing how to modify a non-identity of $\widetilde
{A_0}$ to  another non-identity whose algebraic operations leave us
in the same quiver.

As stated in the introduction, we need to provide hiking for
quasilinear polynomials.
 The hiking procedure   requires
three different stages.

Actually, $f$ has three kinds of variables which play important
roles:
\begin{itemize}
\item Core variables, used for exclusive absorption inside the
radical (such as variables which appear in commutators with central
polynomials of Wedderburn blocks),\item variables used for hiking,
\item variables inside Capelli polynomials used for computing the
actions of characteristic coefficients.
 \end{itemize}

\begin{exmpl} An easy example of the underlying principle:
If $k = 2$ with $n_1> n_2$, then the quiver $\Gamma$ consists of two
blocks and an arrow connecting them, so we  replace a variable $y$
of~$f$ with a radical substitution, $h_{n_1,1}[h_{n_1,2},z] y
h_{n_2}$. The corresponding specialization remains in the radical.
Here we are ready to utilize the techniques given below in
\S\ref{trab} to compute characteristic coefficients, bypassing the
complications of hiking.
\end{exmpl}


\begin{defn}\label{mol} Given a polynomial $f(x_1, \dots,x_\ell)$  and another polynomial $g$, we write $f_{x_i \mapsto g}$ to
denote that $g$ is substituted for $x_i$. We say that~$f$ is
\textbf{hiked} to $\tilde f : = f_{x_i \mapsto g}$ (at~$x_i$) if $g$
is linear in $x_i$.

We call the replacement $g$ of $x_i$ an \textbf{atom} of the hiked
polynomial. A \textbf{molecule} is the product of atoms.
 \end{defn}

The motivation for hiking is that the hiked polynomial $\tilde f $
lies in the T-ideal of~$f$ but combinatorially we have greater
control over the nonzero substitutions of $\tilde f $.

 Suppose we
have the polynomial~$f$, with a radical substitution. We replace it
and have a hiked polynomial. 
 But to continue, we shall need a
rather intricate analysis.

 \begin{defn}\label{doccr} A polynomial is \textbf{bonded} (of \textbf{length} $d$)
if it can be written in the form
 $$\sum _u g_{u,1}h_{\tilde n}(y) g_{u,2} h_{\tilde n}(y)\cdots g_{u,d} h_{\tilde n}(y)g_{u,d+1}$$
 for suitable polynomials $g_{u,i}$ (perhaps constant) in which the
 $y$ indeterminates do not occur. (In other words the
 $y$ indeterminates  occur only in the $h_{\tilde n}(y)$.) The $ h_{\tilde
 n}$ are called the \textbf{bonds} and are the alternating polynomials where we examine
 substitutions.

 A bonded polynomial $f(x_1, \dots, x_t; y,y',y''; z,z')$ is \textbf{critical} if any nonzero substitution
of the $y_i$ is right.
\end{defn}
Thus the bonds are attached to atoms. If~$f$ is hiked to various
polynomials $f_j$ we also say it can be hiked to $\sum f_j$.

Note that the situation is complicated by the fact that if the $x_i$
repeat
 then the atoms repeat, and thus the variables $y$ repeat.

(Likewise for other indeterminates that appear once the hiking is
initiated.)


\begin{rem} First suppose that the depth $u = k,$ i.e., all $n_j = \tilde n$,
and there are no nonzero external radical substitutions. In other
words, the only nonzero substitutions involve specializing all the
$x_i$ to semisimple elements in blocks of degree $\tilde n$. Then we
simply replace~$f$ by $hf$, which trivially is bonded, and Theorem
\ref{hikthm} is proved. So in the continuation, we assume that
$u<k$, which means there is some nonzero substitution
$f(\overline{x_1}, \dots, \overline{x_k})$ in our dominant branch
$\mathcal B$, for which some $\overline{x_i}$ is an $\tilde
n$-bridge. We pass to this $\overline{x_1}, \dots, \overline{x_k}$
in
what follows, and call it our \textbf{working substitution}.
\end{rem}

%
%

\section{The Hiking Theorem for Polynomials}\label{hik}

In this section we prove the Canonization Theorem for Polynomials,
by means of a more technical version to handle the mixed case.

\begin{thm}[Hiking Theorem for Polynomials]\label{hikthm7} Suppose $f(x_1, \dots, x_\ell) $ is a  
non-identity of $\widetilde {A_0}$, possibly with a mixed or pure
substitution. Then~$f$ can be hiked to a critical nonidentity in
which all of the substitutions of the~$x_i$ in the dominant
branch~$\mathcal B$ are right.
\end{thm}

%
%

%

 The
proof of Theorem~\ref{hikthm7} is through a succession of hiking
steps in order both to eliminate ``wrong'' substitutions and then
bonding, i.e., insert $h_{\tilde n}$  into the polynomial.  The
latter is achieved by replacing $z_i$ by $h_{\tilde n} z_i $ and
$z_i'$ by $z_i'h_{\tilde n} $; i.e., we pass to $f_{z_i \mapsto
h_{\tilde n} z_i,\ z_i' \mapsto z_i' h_{\tilde n} } .$

 The hiking procedure   is performed in
three different stages.

\subsection{Preliminary hiking}\label{firstst}$ $

Our initial use of hiking is to resolve some technical issues.
First, we want to eliminate the effect of $(q_1,q_2)$-Frobenius
gluing for $q_1 \ne q_2$, since it can complicate bonding. Toward
this end, we substitute $z_{i'} c_{n_j}(y)^{q_1/q_2} $ for $z_{i'}$,
for each instance of Frobenius gluing. It makes the Frobenius gluing
identical on~$f$.

%


 We also need the base fields of the components all to be the same.
When $\mathcal B'$ is another branch with the same degree vector,
and the corresponding base fields for the $i$-th vertex of $\mathcal
B$ and $\mathcal B'$ are $n_i$ and $n_i'$ respectively, we take $t_i
= q^{n'_i}$ and replace $x_i$ by $(c_{n_i}^{t_i}-c_{n_i})x_i.$ This
cuts off  specializations to matrices over finite fields of the
wrong order.

\subsection{First stage of hiking}\label{firstst1}$ $

We have a quasi-linear nonidentity~$f$ of the Zariski closed algebra
$\widetilde {A_0}$ for which we have a working substitution in some
branch $\mathcal B$, where  $\overline{x_{i_1}}$ in $\widetilde
{A_0}$  is an $\tilde n$-bridge, corresponding to an edge in the
full quiver whose initial vertex is labeled by $(n_{\ell}, t_{i_1})$
and whose terminal vertex is labeled by $(n_{i_1+1}, t_{i_1+1})$
where $\tilde n = \max \{ n_{i_1}, n_{i_1+1}\}.$ We replace
$x_{i_1}$ by $ c_{ n_{i_1} }z_{i_1} [x_{i_1}, h_ {\tilde  n-1} ]
 z_{i_1+1} c_{ n_{i_1+1}  } $,  (where as always the~$c_{n_{i_1}} $ involve new indeterminates in~$v$), and $z_{i_1},z_{i_1+1}$ also are new indeterminates which we
call ``auxiliary indeterminates''; this yields a quasi-linear
polynomial in which any substitution of $x_{i_1}$ into a diagonal
block of degree $<n_1$ or a bridge which is not an $\tilde n$-bridge
is 0. For each semisimple substitution $\overline{x_{i}}$ in a block
of degree $n_i$,  taking $[x_{i_1}, h_ {n_{i_1}}]$ yields 0. This
removes all semisimple component substitutions in $h$ of such $x_i$
whose degree is too ``small,'' i.e., less than $n_i$. For the time
being, we could still have radical substitutions,
 but the first stage of hiking does prepare
for their elimination in the second stage.

 The number of
extra $\tilde n$-bridges in a specialization of $ c_{ n_{i_1}
}(v)z_{i_1} [x_{i_1}, h_ {n_{i_1}'}(v)]
 z_{i_1+1} c_{ n_{i_1+1}  }(v) $ is called its (first stage) \textbf{bridge contribution}.
  (In other words, one takes the total number of
 bridges, and subtracts 1 if $\overline {x_i}$ is an $\tilde
 n$-bridge.)

\begin{lem}\label{expans15} Any nonzero specialization of $h$ is either $\tilde n$-semisimple, or its bridge
 contribution   is positive.
\end{lem}
\begin{proof}
By definition, if the  bridge contribution   is 0 then every substitution has to be semisimple or a  $(j,j)$-bridge for some $j$.
If  $\tilde n$ does not appear then the graph would
have 
such bridges.
\end{proof}

\begin{lem}\label{expans16}
After the first stage of hiking, a
wrong specialization of an $\tilde n$-semisimple element  cannot be
$m$-semisimple for $m<\tilde n$ unless its bridge contribution is at
least 2.
\end{lem}
\begin{proof} When evaluating $h_{\tilde n}$ on semisimple elements
of degree $m$ we get 0 unless we pass away from the $m$-semisimple
component, which requires two bridges.
\end{proof}

\begin{lem}\label{expans17}
After the first stage of hiking, a
wrong specialization of an $\tilde n$-semisimple element is either
$\tilde n$-semisimple or its bridge contribution is at least 1.
\end{lem}
\begin{proof} When evaluating $h_{\tilde n}$ on semisimple elements
of degree $m$ we get 0 unless we pass away from the $m$-semisimple
component, which requires a bridge.
\end{proof}

Thus, the first stage of hiking does not instantly zero out bridges
for wrong specializations, $\overline{x_{i}'}$, but does prepare for
their elimination in the second stage.

 Appending the Capelli polynomials also  sets the stage for eliminating other unwanted
substitutions in the second stage.

 After repeated applications of first stage hiking, we  wind up with a new polynomial $f(x_1, \dots, x_\ell;v;z)
$ where we still have our original indeterminates $x_i$ but have
adjoined new indeterminates.


\subsection{Second stage of hiking}\label{sechik}$ $

\begin{exmpl} To introduce the underlying principle, here is a slightly more
complicated example. Consider the quiver of three arrows, from
degree 2 to degree 1, degree 1 to degree 1, and finally from degree
1 to degree 1.

 First we multiply on the left by $c_4[h_{1,1},z_1]z_2.$  The second
substitution could have an unwanted position inside the first matrix
block of degree~2, since $c_4[h_{1,1},z_1]z_2$ could be evaluated in
the larger component. We take $f_{x_1 \mapsto c_{2}y'y   x_1} -f_{
x_1 \mapsto x_1 c_{2} y'y} ,$ i.e., we multiply by a central
polynomial $h_2$ on the left and subtract it from a parallel
substitution of $h_2$ on the right. The unwanted substitution then
cancels out with the other substitution and leaves~0.
\end{exmpl}

 In the second stage of
hiking, in the blended case, we arrange for all previously
unassigned nonzero substitutions to be pure radical.
%
%
%
%

Suppose $f(x_, \dots, x_\ell; y;z;z')$ is already hiked after the
first stage, and in the branch $\mathcal B$ the indeterminate $z_i$
occurs of degree $d_i$ and the indeterminate $z_{u+1}$ occurs of
degree~$d'$, where $1 \le j \le u.$

\begin{prop}\label{expans0} There are three cases to consider:
\begin{enumerate}\eroman
\item There is a string $\overline{x_{i-1}}  \overline{x_{i}} \cdots  \overline{x_j} \overline{x_{j+1}}$
where $\overline{x_{i}}, \cdots, \overline{x_j}$ are all semisimple
of the same degree $x_{n_j}$ whereas $\overline{x_{i-1}}
\overline{x_{i}},\ \overline{x_j} \overline{x_{j+1}}$ are both
$\tilde n$-bridges.

We take the polynomial
\begin{equation}\label{OK} f_{z_i \mapsto h_{\tilde n}(y')^{d_i}  z_i}
-f_{ z'_{u+1} \mapsto z'_{u+1} h_{n_u}(y')^{t_i} } ,\end{equation}
where the branch $\mathcal B$ has depth $u$ and $t_i$ designates the
maximal degree of $x_i$ in a monomial of $\mathcal B$, where $y'$ is
a fresh set of indeterminates.

\item There is a string $\overline{x_{1}}  \overline{x_{i}} \cdots \overline{x_j} \overline{x_{j+1}}$
where $\overline{x_{1}} \cdots, \overline{x_j}$ are all semisimple
of the same degree $x_{n_1}$ whereas $\overline{x_{j}}
\overline{x_{j+1}}$ is an $\tilde n$-bridge. We take the polynomial
\begin{equation}\label{OK2} f_{z_1 \mapsto h_{\tilde n}(y')^{d_1}  z_1}
 .\end{equation}

\item There is a string $  \overline{x_{i-1}}   \overline{x_{i}} \cdots  \overline{x_{k}} \overline{x_k}$
where $\overline{x_{i}}, \cdots, \overline{x_{k-1}}$ are all
semisimple of the same degree $x_{n_1}$ whereas $
\overline{x_{i-1}}\overline{x_{i}}$ is an $\tilde n$-bridge. We take
the polynomial
\begin{equation}\label{OK3} f_{z_1 \mapsto h_{\tilde n}(y')^{d_1}  z_1}
 .\end{equation}
 \end{enumerate}

This procedure was described so far without Frobenius twists. To
eliminate superfluous Frobenius twists, we also perform the
substitutions
\begin{equation}\label{sechik1} h_{\tilde n}(y')^{q_1} x_j - x_j  h_{\tilde n}^{q_2}(y'),\end{equation}
where $q_1\ne q_2$ range over the various powers of $p$.

 This hiking zeroes out semisimple substitutions of highest degree
(namely $\tilde n$), but not a radical substitution at the $u$ block.
\end{prop}

\begin{proof} (Note that (i) is the usual case, but we also need (ii)
and (iii) to handle terms lying at either end of the polynomial.)
The expression \eqref{OK} yields zero on a semisimple substitution,
but not on a radical substitution, since exactly one of the two
summands of \eqref{OK} would be 0. Likewise   the other cases  yield
zero on a semisimple substitution, but not on a radical
substitution.

Multiplying by all substitutions of \eqref{sechik1}  annihilates all
non-identity Frobenius twists.
\end{proof}

%
%
%


%

\begin{lem}\label{expans1} The second stage of hiking forces any
nonzero specialization of an $\tilde n$-bridge also to be an $\tilde
n$-bridge.
\end{lem}
\begin{proof} In order to provide a nonzero value, at least one of its
vertices must be of degree~$\tilde n$. But if both were  $\tilde n$
the evaluation would be  $ 0,$ by
Lemmas~\ref{expans15}---\ref{expans17} and Remark~\ref{expans0}.
Thus we get an $\tilde n$-bridge.
\end{proof}

\begin{lem}\label{expans2} After the first and second stages of
hiking, the positions of semisimple substitutions of degree~$\tilde
n$ in nonzero evaluations are fixed; in other words, semisimple
substitutions of degree~$\tilde n$ are $\tilde n$-right.
\end{lem}
\begin{proof} Lemma~\ref{expans1} ``uses up'' all the places for $\tilde n$-bridges,
since more $\tilde n$-bridges would yield a substitution
contradicting the maximality of the number of $\tilde n$-bridges in
$\mathcal B$. If $\mathcal B$ has no semisimple substitutions of
degree $\tilde n$ then there is no room for any semisimple
substitutions of degree $\tilde n$ whatsoever, and we are done.

But if  $\mathcal B$ has a semisimple substitution  of degree
~$\tilde n$, that substitution must border an $\tilde n$-bridge,
fixing the order of the pair of indices in the $\tilde n$-bridge,
and thus fixing the positions of all the gaps of index $\tilde  n$
between $\tilde n$-bridges, so again we are done.
\end{proof}

\begin{rem}\label{fin} Although this is already taken care of in the proof, we could have removed finite components
simply by substituting $x_i^m - x_i^\ell$ for $x_i$, for suitable
$\ell,m$.\end{rem}

\subsection{Conclusion of the proof of the Hiking and Canonization Theorems for Polynomials}$ $

 {\bf Proof of the Hiking Theorem for Polynomials
(Theorem~\ref{hikthm7}).} Just iterate the hiking procedure down from
$\tilde n$. (It might well be that the right substitutions to degree
$\tilde n$ cancel, cf.~Remark~\ref{dombr}(2), and then we continue
to $\tilde n -1$); when we finally get to 1 then we are in the
unmixed case, which was handled in \S\ref{nonmixed}.   $\square$

\bigskip

 {\bf Proof of the Canonization Theorem for Polynomials (Theorem~\ref{hikthm}).} After finishing the hiking, one obtains the bond
 (with a nonzero specialization) by replacing~$f$ by
$$f_{z_u \mapsto c_{\tilde n}(y')^{t_u} z_u, \ z'_{u+1} \mapsto
z'_{u+1} c_{n_u}(y')^{t_1}}.$$ \hfill $\square$

%
%

\begin{exmpl}\label{414}
Let us run through the hiking procedure, taking $$\widetilde {A_0} =
\left\{\left(\begin{array}{ccccc}
* & * & * & * & *\\ %
* & * & * & * & *\\ %
* & * & * & * & *\\ %
0 & 0 & 0 & * & * \\
0 & 0 & 0 & * & *
\end{array}\right)\right\}.$$
We have the full quiver $$I_3 \to \II_2,$$ and take the nonidentity
$f = x_1 [x_2, x_3]x_4 + x_4[x_2, x_3]x_1^2.$ We denote the second
matrix component,   as $B = \M[2](K)$,

 We have nonzero
specializations with $\overline{x_1}$ in the first matrix component,
$\overline{x_2}$ an external radical specialization, and
$\overline{x_3}$ in  $B$, but also we have a nonzero specialization
of all variables into $B$. To avoid this situation, we replace~$f$
by
$$f(c_3(y)zx_1z',x_2,x_3,x_4) = c_3(y) zx_1z' [x_2, x_3]x_4+
x_4 [x_2, x_3]^2 c_3(y) zx_1 z' c_3(y) zx_1 z' .$$  Now any specialization into $B$ becomes 0, so
we have eliminated some ``wrong'' specializations. For stage 2 we
take $$\begin{aligned} \tilde f( x;y;y'; z;z')& : =
f(c_3(y)c_3(y')^2 zx_1z',x_2,x_3,x_4) - f(c_3(y)
zx_1z'c_3(y'),x_2,x_3,x_4) \\ = (c_3(y)& c_3(y')^2 zx_1z' [x_2,
x_3]x_4+
 x_4 [x_2, x_3]^2 c_3(y) c_3(y')^2 zx_1 z' c_3(y) zx_1 z')\\ & - (c_3(y) zx_1z' c_3(y') [x_2,
x_3]x_4+
 x_4 [x_2, x_3]^2 c_3(y) zx_1 z' c_3(y') zx_1 z'c_3(y')) ,\end{aligned}$$
 where we see the specialization of highest degree in the first matrix component has
 been eliminated. We can eliminate the nonzero specializations of
 $c_3(y'')$ of degree 1 by taking
 $\tilde f( x;y;y';c_3(y'') z;z')- \tilde f( x;y;y'; z;z'c_3(y''))$
 which leaves us only with a radical specialization and a
 critical polynomial with a single
 bond $c_3(y'')$.

 Note how quickly the polynomial becomes complicated even though we
 have hiked only one of the original indeterminates.
\end{exmpl}

\begin{rem} Other examples of hiking are given in \cite{BRV5}. The main difference between the hiking procedure of this
paper and that of stage 3 hiking of \cite{BRV5} is in the treatment
of the Frobenius automorphism. Stage 4 hiking from \cite{BRV5} is
also analogous.
\end{rem}

\section{Removing ambiguity of matrix degree}$ $

Any polynomial~$f$ of Theorem~\ref{hikthm} can be written as a sum
of homogeneous components $\sum f_j$, where $f_j$ has the same
matrix degree for each monomial. We want to reduce to homogeneous
components.

\begin{defn}\label{iso} A hiked polynomial is \textbf{uniform}
if there is some indeterminate $x_i$ for which, in each of its
monomials, the atom obtained from hiking $x_i$ is semisimple of the
same matrix degree.
\end{defn}

Our objective in this section is to hike to a uniform polynomial.
First we use \S\ref{nonmixed} to dispose of the easy case where each
hiked monomial has a semisimple atom (Definition~\ref{mol}).

\begin{defn}\label{iso1} A radical element of a molecule is \textbf{isolated}
if multiplication by any radical element on the left or right is
zero. \end{defn}

\begin{rem}\label{unhik} The product of two isolated elements is 0, by
definition.
\end{rem}

\begin{prop}\label{closed1} Any polynomial with a nonzero substitution can be hiked to a uniform polynomial, with a nonzero substitution.
\end{prop}
\begin{proof} Multiply $x_i$ by a new indeterminate $x_i'$ and hike that without making the substitution zero.
We are done unless it provides a radical substitution. Since
$J^{t+1} = 0$, we get an isolated element after at most $t$ hikes.
An extra occurrence of $x_i$ which is hiked on must then be
semisimple.
\end{proof}


\begin{lem}\label{ZarS0} We may hike further so that all matrix components  of
size $\tilde n$ are defined over the same field.\end{lem}
\begin{proof}
In Proposition~\ref{closed1} we have just reduced to the case where
all monomials have atoms of some $x_i$ of the same matrix degree
$\tilde n$ (and the substitutions of the $x_i$ are all semisimple),
but we next must contend with the possibility that the different
matrix components may be defined over different fields. But these
all have the same characteristic $p$ (the characteristic of $F$), so
have sizes say $p^{t_{i_1}}$ and $p^{t_{i_2}}$, and we modify
Proposition~\ref{expans0} by applying the appropriate Frobenius maps
$x \mapsto p^{t_{i_j}}$ at the various bonds.
\end{proof}

\section{Characteristic coefficient-absorbing polynomials inside
T-ideals}\label{trab}

Having started with our T-ideal $\mathcal I$ and a $\widetilde
{A_0}$-quasi-linear polynomial $f\in \mathcal I$
 with a nonzero evaluation (where we identify  a representation $\widetilde {A_0}$ with the full quiver of~$f$),
 we have seen how to hike~$f$ in various stages to get specific properties and still have a nonzero substitution.
Utilizing all of these hiking procedures and replacing $\mathcal
I$ by the T-ideal generated by this hiked polynomial, we may make
the following assumptions
 on~$f$:
\begin{itemize}
\item All monomials of~$f$ have the same matrix degree $\tilde n$, and over the same finite base field,
although the multiplicities might vary because of gluing;
\item~$f$ is uniform, so a radical substitution   into an auxiliary indeterminate~$z$ inside
$f$   yields 0; hence we have a \textbf{substitution action} of
semisimple elements, preserving $\mathcal I$ (the T-ideal generated
by the substituted polynomial is obviously contained in the T-ideal
$\mathcal I$ generated by~$f$).
\end{itemize}

To understand this substitution action, we want to utilize the
well-understood properties of semisimple matrices (especially the
coefficients of their characteristic polynomials, which we call
\textbf{characteristic coefficients}). We follow the treatment of
coefficient-absorbing polynomials from \cite[Theorem~4.26]{BRV5} and
\cite[\S 6.3]{BRV6}, although we can skip much of the discussion
there because we already have obtained a bonded polynomial (see \Dref{doccr}).

%
%
Using Theorem~\ref{hikthm}, we work with quasi-linear polynomials
and pinpoint semisimple substitutions of degree $\tilde n$, in order
to utilize the well-understood properties of semisimple matrices
(especially the characteristic coefficients of their simple
components).

\subsection{Characteristic coefficients}$ $

Over a field $K$, any matrix $a \in \M[n](K)$ can be viewed either
as a linear transformation on the $n$-dimensional space $V =
K^{(n)}$, and thus having Cayley-Hamilton polynomial~$f_a$ of
degree~$n$, or (via left multiplication) as a linear transformation
$\tilde a$ on the $n^2$-dimensional space $\tilde V = \M[n](K)$ with
Cayley-Hamilton polynomial~$f_{\tilde a}$ of degree~$n^2$. The
matrix $\tilde a$ can be identified with the matrix $$a \otimes I
\in \M[n](K) \otimes \M[n](K) \cong \M[n^2](K),$$ so its eigenvalues
have the form $\beta \otimes 1 = \beta$ for each eigenvalue $\beta$
of $a$. But there are only finitely many components in the
representation, so $\tilde a$ is algebraic in $M_{\tilde n}(K).$

We recall a basic observation of Zariski and  Samuels:

\begin{lem}\label{ZarS}
Any characteristic
coefficient of an element which is integral over a commutative ring $C$,
is itself integral over $C$.
\end{lem}
\begin{proof} If $\a $ denotes the characteristic
coefficient, then $\a $ is generated by powers of roots of the
minimal polynomial of the given element.
\end{proof}

 From this, we
conclude:

\begin{prop}[{\cite[Proposition~2.4]{BRV3}}]\label{obv2}
Suppose $a \in \M[n](F)$. Then the characteristic coefficients of
$a$ are integral over the $F$-algebra ${\mathcal C}_a$ generated by
the characteristic coefficients of $\tilde a$.
\end{prop}
\begin{proof}
The integral closure of ${\mathcal C}_a$ contains all the
eigenvalues of $\tilde a,$ which are the eigenvalues of $a,$ so the
characteristic coefficients of $\tilde a$ also belong to the
integral closure.
\end{proof}

We are about to adjoin (finitely) many integral elements. Recall
from \Rref{Noethind} that $A_0$ is a ``minimal representable cover"
of the algebra~$A$. We can define the characteristic coefficients
 via polynomials.

%

%
%
%


\begin{defn}\label{absorp} Given a quasi-linear polynomial $f(x;y)$ in indeterminates
labeled~$x_i,y_i$, we say~$f$ is {\bf characteristic
coefficient-absorbing} with respect to its full quiver~$\Gamma$ if
the linear span of $f(\widetilde {A_0})$
 absorbs multiplication by any characteristic coefficient of any
 element in each bonded (diagonal)
 matrix block. 
 \end{defn}

\begin{rem}\label{CHrad} In view of  Proposition~\ref{closed1} we may use characteristic
coefficients for semisimple elements. We recall that we are working
in characteristic $p>0.$ In order to guarantee that the semisimple
substitutions  are indeed semisimple as matrices, we take the Jordan
decomposition of the matrix $a = s+r$ where $s$ is semisimple and
$r$ is nilpotent with $sr = rs,$ and then observe that if $r^k = 0$
and $\bar q$ is a $p$-power greater than~$k$, then
$$a = (s+r)^{\bar
q} = s^{\bar q} + r^{\bar q} = s^{\bar q}+0 = s,$$ which is
semisimple. This leads us to take ${\bar q}$-powers of matrices, and
$\bar q$-powers of characteristic coefficients.
\end{rem}

\begin{lem}[as in {\cite[Lemma~3.6]{BRV4}}]\label{q1}
Write an $\tilde n$-alternating polynomial~$f$ as a sum of
homogeneous components $\sum f_j$. Each $f_j$ is characteristic
coefficient absorbing in the blocks of degree $\tilde n$.
\end{lem}
\begin{proof}
The proof can be formulated in the language of \cite[Theorem~J,
Equation~1.19, page~27]{BR} (with the same proof), as follows,
writing $T_{a,j}$ for the transformation given by left
multiplication by $a$:
\begin{equation}\label{traceab0}
\alpha_k f_j(a_1, \dots, a_t, r_1, \dots, r_m) = \sum
f_j(T_a^{k_1}a_1, \dots, T_a^{k_t}a_t, r_1, \dots,
r_m),\end{equation} summed over all vectors $(k_1, \dots, k_t)$ with
each $k_i \in \{ 0, 1\}$ and $k_1 + \dots + k_t = k,$ where
$\alpha_k $ is the $k$-th characteristic coefficient of a linear
transformation $T_a: V \to V.$
\end{proof}


%

Since the sole purpose of the hypotheses of  Lemma~\ref{q1} was to
obtain the conclusion~\eqref{traceab0}, we merely
assume~\eqref{traceab0}.
\begin{lem}\label{q2} For any  polynomial $f(x_1, x_2,
\dots)$ quasi-linear in $x_1$ with respect to a matrix algebra
$\M[n](F)$, satisfying \eqref{traceab0}, there is a homogeneous
component $\hat f $ in the T-ideal generated by~$f$ which is
characteristic coefficient absorbing.
\end{lem}
\begin{proof} Take the polynomial of Lemma~\ref{q1}, after we zero
out the substitutions of all but one of the components.
\end{proof}

\begin{rem}\label{CHid}
Notation as in \eq{traceab0}, where $f = f_j$, the Cayley-Hamilton
identity for $n \times n$ matrices is
$$\begin{aligned} 0 & = \sum_{k=0}^{n} (-1)^k
\alpha_k f(a_1, \dots, a_t, r_1, \dots, r_m) \la ^{n-k}  \\ & =
 \sum_{k=0}^{n} (-1)^k \sum_{k_1+\dots+k_t= k} f(T_a^{k_1}a_1, \dots, T_a^{k_t}a_t, r_1, \dots,
r_m) \la ^{n-k},\end{aligned}$$ which is thus an identity in the
T-ideal generated by~$f$.
\end{rem}

\begin{defn}\label{HCind}
We call this identity $$\sum_{k=0}^{n} (-1)^k \sum_{k_1+\dots+k_t=
k} f(T_a^{k_1}a_1, \dots, T_a^{k_t}a_t, r_1, \dots, r_m) \la
^{n-k}$$   the {\bf Cayley-Hamilton identity induced by~$f$}.
\end{defn}

\subsection{Controlling the action of characteristic
coefficients}\label{contact}$ $

\begin{defn}\label{6.7}
Fixing $0 \le k <n,$ we denote the $k$-th characteristic coefficient
of $a$, as $\aqpol(a)$.
\end{defn}

Now we  use hiking   to force the polynomial-defined characteristic coefficients of
the matrices to commute with each other.

\begin{prop} \label{modhike1}
One can hike~$f$ such that the characteristic coefficients
$\aqpol(a)$ of any matrix evaluation commute with each
other.\end{prop}
\begin{proof} Take  homogeneous $\hat f$ of
Lemma~\ref{q2}, and one more indeterminate $y''$. There is a Capelli
polynomial $\tilde c_{n_i^2}(y''): = \tilde c_{n_i^2}(y'', \dots)$
and $p$-power $\bar q$ such that
\begin{equation}\label{CC=CC}
\tilde c_{n_i^2}( \a_k y'') x_i c_{{n'_i}^2}(y'') = \a_k ^{\bar
q}(y_1)c_{n_i^2}(y'') x_i c_{{n'_i}^2}(y'')\end{equation} on any
diagonal block. Since characteristic coefficients commute on any
diagonal block, we see from this that
\begin{equation}\label{hikemore}\tilde c_{n_i^2}(y'') x_i
c_{{n'_i}^2}(y'')\tilde c_{n_i^2}(z) x_i c_{{n'_i}^2}(z) - \tilde
c_{n_i^2}(z) x_i c_{{n'_i}^2}(z) \tilde c_{n_i^2}(y'') x_i
c_{{n'_i}^2}(y'')\end{equation} vanishes identically on any diagonal
block, where $z = \a_k y''$. One concludes from this that
substituting~\eqref{hikemore} for $x_i$ would hike $\hat f$ one step
further. But there are only finitely many ways of performing this
particular hiking procedure. Thus, after a finite number of hikes,
we arrive at a polynomial in which we have complete control of the
substitutions, and the characteristic coefficients defined via
polynomials commute.
\end{proof}

\begin{notation}\label{not63}

Let $\mathcal S$ denote the finite set of products  (of length up to
the bound of Shirshov's Theorem~ \cite[Chapter 2]{BKR}) of components
according to the Peirce decomposition (sub-Peirce components when
considering rings without 1) of the generic generators of $A_0$.

 Let  ${\mathcal C}$ be the algebra obtained by adjoining to $F$
 the
 characteristic coefficients of the elements of $\mathcal S$, and
$\widehat {A_0}=   \widetilde {A_0} \otimes _F {\mathcal C},$  the
algebra obtained by adjoining these characteristic coefficients  to
$\widetilde {A_0}$.

 We introduce a commuting indeterminate
$\la_i$ for each of these finitely many characteristic coefficients
$\alpha _i,$ $i \in I$, define $\mathcal{C}'$ to be ${{\mathcal
C}}\,[\la_i : i \in I]$, and $A' $ to be ${\widehat{ A_0}}\,[\la_i :
i \in I]$.
\end{notation}

In this way, after hiking, the substitution action now is
well-defined on any single
 monomial of our hiked polynomial~$f$, and is integral over the action
involving a single component. Our total degree of integrality could be huge (although it is bounded, since the degree for each element is
bounded).

\subsection{Closed submodules}\label{cltrab}$ $

\begin{defn}\label{trmat00}
We take
$\mathcal{C}$ and $\widetilde {A_0}$ from Notation~\ref{not63}, noting that  ${\mathcal C}$ is central.
An ideal $U$ of $\widetilde {A_0}$ is \textbf{closed} if
$\mathcal{C} U = U$, i.e. if $U$ absorbs multiplication by elements
of ${\mathcal C}$.
\end{defn}

We just saw in principle how to obtain closed T-ideals using a
dominant branch, but we have to contend also with pseudo-dominant
branches from~$f$. This is tricky, since we must contend with
different degrees of $\tilde n$. Even in the homogeneous case,
different monomials might define different actions, so we do not
have a single action that we can mod out.  But, even worse, for
nonhomogeneous polynomials~$f$,
 different components of the same matrix degree
$\tilde n$ may occur with different multiplicities in the monomials
of the relatively free algebra $A$, so the characteristic
coefficient arguments of Proposition~\ref{modhike1} may work
differently for  different monomials of~$f$. In order to identify
these actions it is conceptually clearer to bring in another
viewpoint of the coefficients of the characteristic polynomial, and
unify it with these other actions.

 The T-ideal~$\CI$
generated by the hiked polynomial~$f$ contains a nonzero T-ideal
which is also an ideal of the algebra~$\widehat{ A_0}$.
 In principle, we shall use Shirshov's Theorem~
\cite[Chapter 2]{BKR} to produce closed ideals, in order to extend
our representation of $\widehat{ A_0}$.
 In view of Shirshov's
theorem we only need to adjoin a finite number of elements to
obtain~${\mathcal C}.$

\subsubsection{Symmetrized characteristic
coefficients}$ $

Our discussion in adjoining characteristic coefficients   involves
ambiguities arising from different branches. We could bypass these
by making identifications in Definition~\ref{idntif} below, but it
seems clearer to identify everything with the following notion, in
view of Lemma~\ref{symcoef} below.

\begin{defn}\label{sym} Given matrices $a_1, \dots, a_t,$ the \textbf{symmetrized} $(k;j)$ characteristic coefficient is the $j$-elementary symmetric function applied to the $k$-characteristic
coefficients of $a_1, \dots, a_t.$ \end{defn} For example, taking
$k=1$, the symmetrized $(1,j)$-characteristic coefficients $\a_t$
are
$$\sum_{j=1}^t \tr(a_j),\quad \sum_{j_1 <j_2} \tr     (a_{j_1})\tr(a_{j_2}),\quad \dots , \quad \prod_{j=1}^t \tr(a_{j}).$$

\begin{lem}\label{symcoef} Any characteristic coefficient $\a _k$ is integral over the
ring with all the symmetrized characteristic coefficients
adjoined.\end{lem}
\begin{proof} If $\a _{k,j}$ denotes the $(n;j)$-characteristic
coefficient, then $\a _k$ satisfies the usual polynomial $\la ^ n+
(-1)^j \sum _{j=1}^n \a _{k,j} \la ^{n-j}.$
\end{proof}

\subsubsection{Computing the action of characteristic
coefficients}\label{contact1} $ $
%
%

If the vertex corresponding to $r$ has matrix degree $n_i$, taking
an $n_i \times n_i$ matrix $w$, we define ${\aqpol}_u(w)$ as in the
action of Definition~\ref{6.7} and then the left action
\begin{equation}\label{mtr1}
a_{u,v} \mapsto {\aqpol}_u (w) a_{u,v}.
\end{equation}
Likewise, for an $\tilde n \times \tilde n$ matrix $w$ we define the
right action
\begin{equation}\label{mtr2} a_{u,v} \mapsto {\aqpol}_v (w) a_{u,v} .
\end{equation}
(We only need the action when the vertex is non-empty; we forego the
action for empty vertices.)

This action repeats according to the multiplicity $m$ of the vertex,
and so for each given $u,v$ we take the left and right
multiplication operators $$ \phi^\ell: a_{u,v} \mapsto {\aqpol}_u
(w)^{m}
 a_{u,v}
 , \qquad \phi^r : a_{u,v} \mapsto {\aqpol}_v
(w)^{m}
 a_{u,v}$$ inside the endomorphism algebra of the
module containing all the substitutions in dominant branches. This
gives us an action on the right substitutions of branches.

We need to consider the endomorphism algebra and its invariants, in
order to cope with possible cancellation in symmetric expressions in
quasi-linearizations.

\subsubsection{Resolving ambiguities}\label{stillfield}$ $

The difficulty is that the action $a \mapsto \phi(a)$ is not
multiplicative in general, i.e. $\phi(a b)$ need not be
$\phi(a)\phi(b).$

%
%

  We need to find a Noetherian module in whose endomorphisms $A $ can
be represented via the Cayley-Hamilton theorem. We need to
coordinate two differing actions. Towards this end we  introduce an
auxiliary ring.

\begin{defn}\label{trmat0}
In the matrix ring $\M[n](K)$, we define
\begin{equation}\label{trmat}
\alphajk(a) : = \sum_{j=1}^n \sum e_{j,i_1} a e_{i_2,i_2}a \cdots a
e_{i_{k}i_k} a e_{i_1,j},\end{equation} the inner sum taken over all
index vectors of length $k$.

The operator algebra generated by multiplication by elements
$\{\alphajk:   \  k \le n\}$ over~$F$ (cf.~Notation~\ref{not63}) is
denoted $\mathcal T$.
\end{defn}

Thus $\alphajk(a)$ gives us the matrix evaluation of a
characteristic coefficient, and $\mathcal T$ provides the
characteristic coefficients.

We have two versions of characteristic coefficients, one given in
Definition~\ref{trmat0} and the other in Remark~\ref{CHid}, but the
matrix version is not necessarily compatible with polynomial
evaluations.

Since we may be in nonzero characteristic, in the main situation our
quasi-linear hiked polynomials need not be homogeneous, and we also
turn to the pseudo-dominant components.

%

\begin{defn}\label{idntif}
Take generators $\psi_1, \dots, \psi_l$ of $\mathcal T$, 
and formally define relations
$$\bar\psi_j^k = \lambda_{j,0} + \lambda_{j,1}\bar\psi_j + \dots +
\lambda_{j,k-1}\bar\psi_j^{k-1}$$ for commuting indeterminates
$\lambda_{j,u}.$

Taking this big ring $A'$ of Notation~\ref{not63},
Let $M = A'/\mathcal J$ where $\mathcal J$ is the ideal
generated by the elements $(\lambda_{j,j'}-\gamma_{j,j'})d$ for $d$
isolated (see \Dref{iso}).
\end{defn}

\begin{lem} $\mathcal J \cap \mathcal I = 0.$ Hence $\mathcal I $
embeds naturally into $M$.
\end{lem}
\begin{proof}
If $f \in \mathcal J \cap \mathcal I,$ then its isolated
substitutions must be 0, but by hypothesis~$f$ has nonzero isolated
substitutions.
\end{proof}

 Let~$W$ be the annihilator of~$\mathcal{J}$ in $A'$.
Then $W\mathcal{J} = 0,$ implying:

\begin{lem}\label{Ahatisfinite1}
The action on $M$ is the same over $\mathcal T$ and $\mathcal C$.
\end{lem}
\begin{proof} The only possible discrepancy is isolated, so the difference in the action comes from the $\lambda_{j,j'} - \gamma_{j,j'}$, which are in $\mathcal{J}$ by definition.
\end{proof}

\begin{lem}\label{Ahatisfinite}
The algebra $M$ is a finite module over~$\,{\mathcal C}$, and in
particular is Noetherian and representable.
\end{lem}
\begin{proof}%
Indeed, $M$ is
a finite module over ${\mathcal C}'$ in view of Shirshov's
Theorem. But ${\mathcal C}'$ is finite over ${\mathcal C},$ in view
of Lemma~\ref{ZarS}, implying $M$ is finite over~${\mathcal C}$.
 Thus $M$ is Noetherian, and is representable by Anan'in's Theorem \cite{An}.
\end{proof}

\begin{lem}\label{Ahatisfinite17}
For any polynomial~$f$, each of its pseudo-dominant branches
provide finite (Noetherian) $\mathcal C$-submodules of $\mathcal T$.
Consequently, $\mathcal T$ is integral and finite over $\mathcal C$.
\end{lem}
\begin{proof}
The first assertion is by Lemma~\ref{Ahatisfinite}. The second assertion follows since the
elements of $\mathcal T$ are integral, and we only need finitely
many to generate $\mathcal T$.
\end{proof}
Also recall that $\mathcal C$ is finite over $F$.

\subsection{Conclusion of the proof of Theorem~\ref{4.66}} \label{CIO}$ $

As mentioned earlier, we assume that $\cha(F) >0,$ since the result
is known in characteristic 0.

%
We have reduced to the case that~$f$ is $A$-quasi-linear and
suitably hiked,  picking one homogeneous component and zeroing out
the other ones.

Now $\CI $ contains a nonzero
 T-ideal $\CI_1$ of $\widehat {A_0}$ generated by $\bar{q}$-characteristic
coefficient-absorbing polynomials of $\CI$ in $ {\mathcal C}'
\widetilde {A_0}.$ The ideal $\CI_1$ is representable by
Lemma~\ref{Ahatisfinite}, implying $\langle f \rangle_T \cap \CI_1$ is
representable, a contradiction by \Lref{Phoenix}.

%
This concludes the proof of Theorem~\ref{4.66}.
%


\section{Proof of Theorem~\ref{4.67}, over an arbitrary Noetherian ring
~$C$}\label{467}

We introduce new notation for the remainder of the paper. Let $A$ be
a given relatively free affine PI-algebra over an arbitrary
commutative Noetherian ring $C$.
 From now on let $J$ denote the nilpotent radical of
the Noetherian ring $C$. We take $t$ maximal such that $J^t \ne 0,$
i.e., $J^{t+1} = 0.$ Write $J$ as a finite intersection $P_1 \cap
\dots \cap P_j$ of prime ideals, with $j$ minimal possible. We call
$j$ the \textbf{irredundancy index} of $J$. The proof is based on a
triple induction in the following order: Specht induction on $A$,
Noetherian induction on $C$, and usual induction on the irredundancy
index.

\subsection{Various aspects of torsion}\label{tors}$ $

The difference for algebras over a Noetherian ring $C$ from the
field-theoretic case is that modules over  $C$ may have torsion. 

Define $F := C/J$. If $C$ is local then $F$ is a field, and we shall
see how to reduce to $F$-algebras. Thus, reduction to $C$ local is
a crucial part of the proof.

\subsubsection{Reduction to all torsion of $C$ contained in $J$}$ $

 We define $\q $ to be the
set of elements of $A$ which have annihilators in $C\setminus J.$ We
claim that we can reduce to the case that $\q =0.$ Assume otherwise
that $\q \ne 0.$

Given $0 \ne a \in \q ,$ define $S_a = \{ c \in C\setminus J : c ^k
a = 0 \text{ for some }k\}$. We want to reduce to the case that $S_a
= \emptyset.$ For $c\in C$, define $\mathcal I_{k,c}= \{ a : c^k a =
0\}$.

\begin{lem}\label{existsn} Given $c\in C$,  there is some $n$ such that if $c \in S_a$  then $c^n a =
0$.\end{lem}
\begin{proof} $\mathcal I_{k,c}$ is clearly an ideal of
$A$, and is a T-ideal, since for any endomorphism $\varphi$ of $A$,
$ c^k \varphi( a) = \varphi(c^k a) = \varphi(0) = 0 .$ Hence, we
have an ascending chain of T-ideals $\mathcal I_{1,c} \subseteq
\mathcal I_{2,c} \subseteq \dots,$ which must stabilize at some
$\mathcal I_{n,c}$. This means that any element annihilating a power
of $c$ must annihilate $c^n$.
\end{proof}

\begin{lem}\label{existsn2} For any $0 \ne a \in \q $, if there is a counterexample $A$ to
Theorem \ref{4.66}, then there is a counterexample $\bar A$ which is
a homomorphic image of $A$ with $S_{\bar a} = \emptyset,$ where
$\bar a$ is the image of $a$ in $\bar A$.
\end{lem}
\begin{proof} Take $A$ a Specht
minimal counterexample. Pick $c \in S_a$, and $\mathcal I_{n,c}$ as
in Lemma~\ref{existsn}. $\mathcal I_{n,c}\ne 0,$ in view of
Lemma~\ref{existsn}. We have an injection $A \hookrightarrow A/
\mathcal I_{n,c} \oplus (A \otimes _C (C/c^n C))$. But $A/\mathcal
I_{n,c}$ is representable, by Specht induction, and $A \otimes _C
(C/c^n C)$ is representable by Noetherian induction on $C$. Hence
$A$ is representable, contrary to assumption. Since this holds for
every $c \in S_a$, we conclude that $S_a = \emptyset.$
\end{proof}

\begin{prop}\label{existsn3} If there is a counterexample $A$ to
Theorem \ref{4.66}, then there is a counterexample $\bar A$ for
which all elements of $C$ making elements of $A$ torsion lie in $J$.
\end{prop}
\begin{proof} Take $A$ a Specht
minimal counterexample. We claim that $\q  = 0.$ Indeed otherwise we
can take $a \ne 0$ in $A$ and contradict Lemma~\ref{existsn2}.
\end{proof}

\subsubsection{Reduction to irredundancy index of $J$ equaling 1}$ $

\begin{cor}\label{existsn47} If there is a counterexample $A$ to
Theorem \ref{4.66}, satisfying the conclusion of
Proposition~\ref{existsn3}, then there is a counterexample in which
the irredundancy index of $J$ is 1.
\end{cor}
\begin{proof} The intersection $P_1 \cap \dots \cap
P_j$ of prime ideals clearly is irredundant. We claim that $j =1.$
Indeed if $j>1$ we can take $s \in P_1 \setminus J$. Localizing $A$
at $s$, we have $A$ embedded into $A[s^{-1}] \oplus (A \otimes _C
(C/s^t C)$. $ A \otimes _C (C/s^t C)$ is representable by Noetherian
induction. Thus, it suffices to show that $A[s^{-1}]$ is
representable. On the other hand, the kernel of the natural map $A
\to A[s^{-1}]$ is $\Ann (s^t)$, implying $C[s^{-1}]/P_i[s^{-1}]
\cong (C /P_i)[s^{-1}]$ for each $i \ge 2$ (for if $s^k (c +P_i) =
0$ then $s^k c \in P_i$ and thus $c \in P_i$). Likewise
$P_i[s^{-1}]$ is a prime ideal of~$C[s^{-1}]$, since if $c'c'' \in
P_i[s^{-1}]$ then some $s^k c'c'' \in P_i,$ implying $ c'c'' \in
P_i$, so $c' \in P_i$ or $c' \in P_i$.

By Lemma~\ref{existsn}, if $c \in P_i$ for $2 \le i \le j$ then $s^n
c \in J$, so $c \in J[s^{-1}].$ On the other hand, if $s^{-k}c$ is
nilpotent then $s^n c = 0,$ implying $c \in P_i$ for $2 \le i \le
j$, and we conclude that $J[s^{-1}]$ is the nilradical of
$C[s^{-1}]$, which has lower irredundancy index. So we are done by
induction on $j$ once we manage to prove the case $j=1$.
\end{proof}

\subsubsection{Reduction to $C$ local}$ $

\begin{prop}\label{existsn4} If there is a counterexample $A$ to
Theorem \ref{4.66}, then there is a counterexample $\bar A$ for
which $C$ is local and all elements of $J$ making elements of $A$
torsion lie in $J$.
\end{prop}
\begin{proof} Take $A$ a Specht
minimal counterexample. By Corollary~\ref{existsn47} we may assume
that the irredundancy index of $J$ equals 1. Localizing by all
elements of $C\setminus J$ we may assume that $J$ is a maximal ideal
of $C$, i.e., $C$ is local.
\end{proof}

\subsection{Further reduction for $C$ Noetherian}$ $

Next, we consider the general case that $C$ is Noetherian; in view
of the previous discussion we may assume that $C$ is a local
Noetherian domain. If $J = 0$ then by Proposition~\ref{existsn4} we
may assume $A$ is torsion free over $C$. Localizing, we embed $A$
into a relatively free algebra over the field of fractions of $C$,
so we could assume that $C$ is a field, which is discussed further
in \S\ref{fieldcase}. Hence we assume that $J \ne 0$, i.e., $t \ge
1$.

There exists $s\in J^{t}$ for which $sA \ne 0,$ since otherwise $J^t
A = 0,$ implying $A$ is a $C/J^{t}$-algebra, and we are done by
Noetherian induction. Take $s\in J^{t}$ for which $sA \ne 0$. $A/
\Ann s$ is an algebra over $F $, since $J \subseteq \Ann s$. But $A/
\Ann s \cong sA$ as modules.

\subsection{Conclusion of the proof via the field case}\label{fieldcase}$ $

We conclude the proof by one last application  of hiking. Take some
polynomial $f \in s A \setminus \{ 0 \}$. The idea is to find a
hiked polynomial in the T-ideal of~$f$, with which we can then apply
Shirshov's theorem to utilize results from integrality.


The module $sA \ne 0$ is a module over $F = A/J$ since $sJ = 0.$
Thus, we can use the theory of hiking on $sA$. Take a nonzero
polynomial $f\in sA$, and let $M$ be the set of hiked polynomials
from~$f$, obtained via Theorem~\ref{hikthm7}. $M$ is a module since
hiking involves a series of four stages of substitutions, and
multiplying a hiked polynomial by another polynomial yields a hiked
polynomial. We take some hiked polynomial $0 \ne g \in M.$ As in
Lemma~\ref{Ahatisfinite}, we can use Shirshov's theorem to adjoin
finitely many elements to $\mathcal C$ to obtain a commutative
ring~${\mathcal C}$ for which
 ${\mathcal C} \otimes gA$ is finite over ${\mathcal C}$ and thus over $\mathcal C$; hence it is representable
by Anan'in's Theorem \cite{An}.

But $gA$ contains a critical nonidentity, which we can then hike to
a critical nonidentity $h$.
 Viewing $hA \subset M_{m}(C')$ we define $\xi_i(a)$ to be the $i$-characteristic
coefficient, i.e., $$a^m = \sum _{i=0}^{m-1} \xi_i(a) a^{m-i} .$$ We
use the module action of \S~\ref{stillfield} in the general,
nonhomogeneous case. Let $\psi_i$ be the substitution operator for
hiked polynomials, and impose the relations
$$(\psi_i(a) - \xi_i(a))h = 0$$ to get a canonical homomorphism $\varphi: A \to \hat A.$
Its kernel intersects $hA$ trivially. This induces a map
 canonical map
$A \to \hat A \oplus (A/hA)$, which is an injection, by
Lemma~\ref{Phoenix}, proving that $A$ is representable.

\section{Appendix: Further applications of hiking, for other categories}\label{invo}

Specht's Conjecture and representability of a T-ideal $\CI$ may be handled in   some other categories  of algebras,   mutatis mutandis, since hiking is a formal process. In this brief appendix we show how to modify the proof for algebras with involution, and indicate how it could also work for other categories. We start by noting that the reduction to  algebras over arbitrary Noetherian goes through as in Section~\ref{467}, which is module-theoretic over the associative commutative base ring. So the issue is for algebras over a field.
A  $C$-algebra in some category is called  \textbf{representable} if there is a 1:1 morphism to
  finite dimensional $K$-algebra  in the category,
 for a suitable field $K$.

\subsection{Algebras with Involution}$ $

An \textbf{involution} of an algebra is an anti-automorphism $(*)$ of order $\le 2$. Involutions occur throughout algebra, in the theory of group algebras and Lie algebras, and more generally, Hopf algebras; matrix algebras with involution play a crucial role
in defining the classical Lie algebras. One develops the theory in terms of $(*)$ in the category, as in \cite[\S 2.13]{Row1.5} and \cite{Row0}. A \textbf{$(*)$-ideal} is an ideal $A$ such that $A^* =A.$ An algebra is $(*)$-\textbf{simple} if it has no proper nonzero $(*)$-ideals. $(C\{x\},*)$ denotes the free associative algebra with involution  in the indeterminates
$x_0,x_0^*, x_1, x_1^*,\dots$, where $(*)$ acts in the obvious way:
$$(x_i^*)^* = x_i; \qquad (x_{i_1}\cdots x_{i_t})^* = x_{i_t}^* \cdots x_{i_1}^*.$$
Its elements $f(x_1,x_1^*, x_2,x_2^*, \dots x_m,x_m^*)$ are called \textbf{$(*)$-polynomials}, and are written here as $f(x_1, x_2, \dots, a_m)$.
$f(a_1,\dots, a_m)$ is the specialization of $f$, under substituting  $x_i\mapsto a_i$ and $x_i^*\mapsto a_i^*.$ For an $F$-algebra with involution $(A,*)$, $f(A,*)$ denotes $\{f(a_,\dots, a_m) : a_i \in A\}.$ We say that $f$ is a \textbf{(*)-identity} of $(A,*)$ if $f(A,*)=0$, and  $(A,*)$ is a $(*)$-\textbf{PI-algebra} if  $(A,*)$  has a nonzero $(*)$-identity. A crucial theorem of Amitsur~\cite{Am} is that every $(*)$-{PI-algebra} is a PI-algebra.

There are three \textbf{standard involutions} related to matrix algebras over a field $K$ of characteristic $\ne 2$.
\begin{enumerate}  \item (Exchange type) $(A,*)=(M_n(K )\oplus M_n(K ) ^{\operatorname{op}},\circ)$, where $(\circ)$ is the exchange involution $(a_1,a_2)^\circ = (a_2,a_1)$. $d^+ = d^- = n^2.$
\item (Orthogonal type) $(A,*)=(M_n(K ),t)$, where $(t)$ is the transpose;  $d^+ = \frac {  n(n+1)}2$ and $ d^- = \frac {  n(n-1)}2.$
\item (Symplectic type) $(A,*)=(M_{2n}(K ),s)$ where $(s)$ is
 $e_{i,j}$:
 $$e_{i,j}^* = e_{j+n,i+n}, \quad e_{i,j+n}^* = -e_{j,i+n},  \quad e_{i+n,j}^* = -e_{j+n,i}
, \quad \forall 1 \le i,j\le n.$$   $d^- = \frac {  n(n+1)}2$ and $ d^+ = \frac {  n(n-1)}2.$
\end{enumerate}

\begin{lem}\label{77}
Any $(*)$-simple algebra over an algebraically closed field can be put into one of these three forms.
\end{lem}
 A  $C$-algebra with involution
$(A,*)$ is called $(*)$-\textbf{representable} if it is embeddable as a
$C$-subalgebra with involution  of a finite dimensional $K$-algebra with involution  $(W,*)$,
 for a suitable  field $K$.  In this case, by means of the regular representation,
one can embed $A$ into some $(M_n(K),*)$ and, tensoring by the algebraic closure of the fixed subfield
of~$K$, assume that  $W = (M_n(K)\oplus M_n(K)^{\circ},\circ)$  or  $W = (M_n(K),*)$ where $(*)$ is a standard involution.

A $(*)$-\textbf{T-ideal} is a T-ideal which is also invariant under $(*)$.

\subsubsection{ACC on $(*)$-{T-ideals} in arbitrary characteristic}$ $

Sviridova \cite{Sviridova} proved the ACC on $(*)$-{T-ideals} (Specht's problem) over a field of characteristic 0; for affine algebras, the key step for affine $(*)$-PI-algebras is
the analog of \cite{Kem11} proved in \cite{Sviridova0} that any affine $(*)$-PI-algebra satisfies precisely the same $(*)$-identities as some finite dimensional $(*)$-algebra, which is essentially the same as the $(*)$-representability of relatively free affine $(*)$-algebras.

For a field $F$  of characteristic $p>0,$ we want to use hiking to obtain these theorems, much as in  \cite{BRV6} and in the main text of this paper. Although the argument has not been published, the solution to Specht's problem was described in \cite{Row2}, which we follow here in developing hiking of $(*)$-polynomials.

Since the   strategy outlined in \cite[Remark~2.3]{BRV6} relied on full quivers on the Zariski closure in a representation of $(A,*),$ we  need the analog for algebras with involution. A (*)-T-ideal $\CI$ is \textbf{representable} if $(C\{x\}/CI,*)$ is a $(*)$-representable algebra.

\begin{rem}\label{prog1}
  The program to prove Specht's conjecture and representability of a (*)-T-ideal $\CI$.
  \begin{enumerate}
  \item
In view of Amitsur's theorem, $\CI$ contains a (*)-T-ideal $\CI_0$ that is representable, so we can work in the  $(*)$-representable algebra $(C\{x\}/CI_0,*)$, which we embed into matrices with over an algebraically closed field $K$,  with standard involution of
transpose or symplectic type. (Exchange type is easily reduced to the non-involutory case by projecting to each component.)
In other words, we may assume that $(A,*) \subseteq (M_n(K),*)$

  \item  The Zariski closure is closed with respect to the involution $(*)$ of  $W$. (Proof: Applying $(*)$ to each
   polynomial relation for $a$ yields a polynomial relation for $a^*$.) Replacing  $(A,*)$ by its Zariski closure,
   we may assume that $(A,*)$ is Zariski closed.

  \item The Jacobson radical of the Zariski closure is a nilpotent $(*)$-ideal.
  \item Any Zariski closed $(*)$-algebra $(A,*)$ can be decomposed into $(S,*)\oplus J,$ where $ (S,*)\cong (A/J,*) $ is a direct product of $(*)$-simple algebras $(S_i,*)$, each of the form of Lemma~\ref{77}. (Indeed, $S$ is the direct product of matrix algebras over fields, since its center is Zariski closed and one mimics the
      classical proof of Wedderburn's theorem in \cite[Theorem 2.5.37]{Row1.5}, as first done by E.~Taft, noting that symmetric idempotents (resp.~antisymmetric) idempotents of $A/J$ lift to  symmetric idempotents (resp.~antisymmetric) idempotents of $A$.
  \item There is a natural $(*)$-version of the Wedderburn block form, and thus a quiver.
 \item If a matrix $a$ is semisimple,
  then so is $a^*$, so quasi-linearization enables us to reduce to semisimple and radical substitutions.
  \item The issues of $(*)$-hiking are the same as noninvolutory  hiking, since one can insert Capelli polynomials without $(*)$.
  \item The $(*)$ version of \Tref{hikthm} is proved in exactly the same way.
    \item One can define characteristic coefficients using either matrices or  polynomials. This is a bit tricky since, for $(*)$ of symplectic type,
    the pfaffian \cite[Theorem~2.5.10]{Row0} takes the place of the characteristic polynomial for symmetric elements, so the degree is $\frac{n}{2}.$
      \item Adjoin the characteristic coefficients using polynomials to get an integral extension, noting that the arguments in the text only used properties of modules over commutative rings.
     \item  Apply Shirshov's theorem to get a $(*)$-algebra finite over a $(*)$-fixed commutative Noetherian algebra, which is $(*)$-representable by the $(*)$-analog of \cite{RowenSmall}.
\end{enumerate}
\end{rem}

\begin{thm}
\end{thm}
\begin{proof}
  In the program of \cite{BRV6}, go through the steps of Remark~2.3, where all the work was done, since hiking behaves in the same way.
Hence  the ACC for $(*)$-{T-ideals} holds over an arbitrary field of characteristic $p>0,$ and in light of Sviridova's theorem, over any field,
and thus over any commutative Noetherian ring.
\end{proof}

\subsubsection{ $(*)$-representability of relatively free affine $(*)$-algebras over a field}$ $

\begin{thm} Any relatively free affine $(*)$-PI algebra over a commutative Noetherian ring is $(*)$-representable.
\end{thm}
\begin{proof}
  Using Specht's $(*)$-Conjecture, we can take a maximal non-$(*)$-representable T-ideal.
Then we repeat the proof of Theorem~\ref{4.66}, using \Rref{prog1} where appropriate.
\end{proof}

\subsection{Other categories}$ $

Some other categories of algebras are amenable to the program in \Rref{prog1}. We need a category for which Specht's problem was solved in characteristic 0,
and for which the simple objects over an algebraically field are easily characterized in terms of their identities, and have matrix-like descriptions in which one can take the Zariski closure and define hiking.

\subsubsection{Alternative algebras}$ $

Alternative algebras have a similar structure theory to associative algebras, in part because 2-generated alternative algebras are associative, by Artin's theorem.
(Of course one takes T-ideals in the free alternative algebra.) Thus any alternative PI-algebra satisfies a 2-generated identity, and \Rref{CHrad} is applicable. Iltyakov \cite{Ilt91} solved Specht's problem for affine alternative algebras of characteristic 0.
Shafer \cite{Schafer} proved the Wedderburn principle theorem, and the only split simple alternative algebras are the split  algebra of octonians (which are algebraic of degree 2 and satisfy the same 2-generated identities of $M_2(F)$)
and the usual associative matrix algebras. Thus one can define a Zariski closure, and  the Wedderburn block form, thereby providing   hiking.

One proves the ACC on T-ideals over an arbitrary field by the induction procedure given in \cite[Definition 7.1 and Lemma 7.2]{BRV6}, and
then the   $(*)$-representability of relatively free affine alternative $(*)$-algebras over an arbitrary field, as in the main text.

\subsubsection{Group-graded algebras}$ $

 Specht's problem was solved for affine algebras of characteristic 0  graded by a finite group,
 by Aljadeff and Belov in \cite{AB}. The gradings in matrix algebras are described explicitly in \cite{BZ}. In characteristic 0, the Jacobson radical often is graded \cite{CM}, but the situation is messier in nonzero characteristic.
  Karasik \cite{Karasik} developed the necessary structure theory in terms of $G$-simple PI-algebras, but their structure seems to be quite complicated, so some details need to be worked out in the Wedderburn block form.

  \subsubsection{Jordan algebras}$ $

The Wedderburn decomposition $S \oplus J$ into split simple Jordan algebras and the radical was discovered by Albert \cite{Albert} for special Jordan
algebras, and by Penico \cite{Penico} in general over a field of characteristic $\ne 2$. One also knows the split simple Jordan algebras, characterized in terms
of their polynomial identities, so
we have  the Zariski closure and Wedderburn block form, and thus can  perform hiking. Vais and Zelmanov~\cite{VZ} proved Kemer's conjecture in characteristic 0, but representability remains open.

\end{document}